\documentclass[11pt,a4paper]{article}
\usepackage{amsmath}
\usepackage{amsfonts}

\newtheorem{theorem}{Theorem}

\newtheorem{corollary}[theorem]{Corollary}

\newtheorem{lemma}[theorem]{Lemma}

\newtheorem{proposition}[theorem]{Proposition}
\newtheorem{remark}[theorem]{Remark}

\newenvironment{proof}[1][Proof]{\noindent\textbf{#1.} }{\ \rule{0.5em}{0.5em}}
\input{tcilatex}

\begin{document}

\title{Harmonic functions which vanish on coaxial cylinders}
\date{}
\author{Stephen J. Gardiner and Hermann Render}
\maketitle

\begin{abstract}
It was recently established that a function which is harmonic on an infinite
cylinder and vanishes on the boundary necessarily extends to an entire
harmonic function. This paper considers harmonic functions on an annular
cylinder which vanish on both the inner and outer cylindrical boundary
components. Such functions are shown to extend harmonically to the whole of
space apart from the common axis of symmetry. One of the ingredients in the
proof is a new estimate for the zeros of cross product Bessel functions.
\end{abstract}

\section{Introduction}

\footnotetext{%
\noindent 2010 \textit{Mathematics Subject Classification } 31B05, 33C10.
\par
\noindent \textit{Keywords: }harmonic continuation, Green function,
cylindrical harmonics, cross product Bessel functions{}}The Schwarz
reflection principle is a beautiful and important result concerning the
extension of a harmonic function $h$ on a domain $\Omega \subset \mathbb{R}%
^{N}$ through a relatively open subset $E$ of $\partial \Omega $ on which $h$
vanishes. The extension is defined by a simple formula, and the domain of
extension is independent of the choice of $h$. When $N=2$ such a reflection
principle holds whenever $E$ is contained in an analytic arc (see Chapter 9
of \cite{Kh96}). When $N\geq 3$ and $N$ is odd, Ebenfelt and Khavinson \cite%
{EK} (see also Chapter 10 of \cite{Kh96}) have shown that a point-to-point
reflection law can only hold when the containing real analytic surface is
either a hyperplane or a sphere. Thus, for other surfaces in higher
dimensions, more elaborate arguments are required to investigate whether
such harmonic extension is still possible.

An important particular case concerns cylindrical surfaces, since a cylinder
is the Cartesian product of a line and a sphere, each of which separately
admits Schwarz reflection. Indeed, prior to the results of \cite{EK}, the
existence of a point-to-point reflection law for cylinders in $\mathbb{R}%
^{3} $ had already been investigated and disproved by Khavinson and Shapiro 
\cite{KS}. Nevertheless, Khavinson asked whether, using $B^{\prime }$ to
denote the open unit ball in $\mathbb{R}^{N-1}$, a harmonic function on the
cylinder $B^{\prime }\times \mathbb{R}$ which vanishes on $\partial
B^{\prime }\times \mathbb{R}$ must automatically have a harmonic extension
to the whole of $\mathbb{R}^{N}$.

This was verified in a recent paper of the authors \cite{GR1}. More
generally, for any $a>0$, it was shown there that a harmonic function on a
finite cylinder $B^{\prime }\times (-a,a)$ which vanishes on $\partial
B^{\prime }\times (-a,a)$ has a harmonic extension to the strip $\mathbb{R}%
^{N-1}\times (-a,a)$. The proof relied on a study of the Green function $%
G_{\Omega }(\cdot ,y)$ for the infinite cylinder $\Omega =B^{\prime }\times 
\mathbb{R}$ with pole at $y\in \Omega $. It is a classical fact that, in
three dimensions, $G_{\Omega }(\cdot ,y)$ can be represented as a double
series involving Bessel functions $J_{n}$ of the first kind of order $n$ and
their zeros, and Chebychev polynomials. In \cite{GR1} such a representation
was established for all dimensions (ultraspherical polynomials take the
place of Chebychev polynomials when $N\geq 4$), and a rigorous analysis of
its convergence properties outside $\Omega $ revealed that $G_{\Omega
}(\cdot ,y)$ possesses a harmonic extension to $\mathbb{R}^{N-1}\times
\left( \mathbb{R}\setminus \left\{ y_{N}\right\} \right) $.

In this paper we turn our attention to the corresponding problem for annular
cylinders. Let $(x^{\prime },x_{N})$ denote a typical point of $\mathbb{R}%
^{N-1}\times \mathbb{R}$ and $\left\Vert x^{\prime }\right\Vert $ denote the
Euclidean norm of $x^{\prime }$. We define 
\begin{equation*}
\Omega _{b}=A_{b}^{\prime }\times \mathbb{R}\text{, \ \ where \ \ }%
A_{b}^{\prime }=\{x^{\prime }:1<\left\Vert x^{\prime }\right\Vert <b\}\text{
\ \ \ }(b>1).
\end{equation*}%
Any harmonic function $h$ on $\Omega _{b}$ that vanishes on the outer
cylindrical boundary component was shown in \cite{GR2} to have a harmonic
extension to the set $\{x^{\prime }:1<\left\Vert x^{\prime }\right\Vert
<2b-1\}\times \mathbb{R}$. We will now establish that considerably more can
be said when $h$ also vanishes on the inner cylindrical boundary component.

\begin{theorem}
\label{main}If $h$ is a harmonic function on $\Omega _{b}$ that vanishes on $%
\partial \Omega _{b}$, then $h$ has a harmonic extension to $\left( \mathbb{R%
}^{N-1}\backslash \{0^{\prime }\}\right) \times \mathbb{R}$.
\end{theorem}

The proof again depends on an analysis of the Green function, but this turns
out to be more challenging for the annular cylinder. Instead of $J_{\nu }$,
the double series expansions now involve factors of the form $J_{\nu }\left(
\rho t\right) Y_{\nu }\left( \rho b\right) -J_{\nu }\left( \rho b\right)
Y_{\nu }\left( \rho t\right) $, where $Y_{\nu }$ is the Bessel function of
the second kind, and the sequence $(\rho _{\nu ,m})_{m\geq 1}$ of positive $%
\rho $-zeros of this expression when $t=1$. Known asymptotic estimates for $%
\rho _{\nu ,m}$ for fixed $\nu $ are insufficient for our purposes, so we
are led to establish a universal lower bound. We use this to show that a
harmonic function on $A_{b}^{\prime }\times (-a,a)$ which vanishes on $%
\partial A_{b}^{\prime }\times (-a,a)$\ must extend harmonically to all of $%
\left( \mathbb{R}^{N-1}\backslash \overline{B^{^{\prime }}}\right) \times
(-a,a)$. It also extends to a specified part of $\overline{B^{\prime }}%
\times (-a,a)$, which increases with $a$. Theorem \ref{main} then follows on
letting $a\rightarrow \infty $.

The proof of Theorem \ref{main} will be developed in Sections \ref{Prep} - %
\ref{endmain}, subject to verification of the estimates for $\rho _{\nu ,m}$%
. These estimates are then established in the final two sections of the
paper.

From now on we will assume that $N\geq 3$.

\section{Zeros of cross product Bessel functions\label{Prep}}

We refer to Watson \cite{Wat} for the definition of $J_{\nu }$ and $Y_{\nu }$%
, the usual Bessel functions of order $\nu \geq 0$ of the first and second
kinds, respectively, and define $N_{\nu }=J_{\nu }^{2}+Y_{\nu }^{2}$.
Further, let $C_{\nu }$ denote any cylinder function of order $\nu $, that
is, $C_{\nu }=\alpha J_{\nu }+\beta Y_{\nu }$ for some $\alpha ,\beta \in 
\mathbb{R}$. We collect below some properties of these functions for later
use.

\begin{lemma}
\label{JV}(i) $\dfrac{d}{dz}z^{\nu }C_{\nu }(z)=z^{\nu }C_{\nu -1}(z)$\ \
and \ $\dfrac{d}{dz}\dfrac{C_{\nu }(z)}{z^{\nu }}=-\dfrac{C_{\nu +1}(z)}{%
z^{\nu }}$.\newline
(ii) $C_{\nu -1}(z)+C_{\nu +1}(z)=\dfrac{2\nu }{z}C_{\nu }(z)$ \ and $C_{\nu
-1}(z)-C_{\nu +1}(z)=2C_{\nu }^{\prime }(z)$.\newline
(iii) $J_{\nu }(t)Y_{\nu }^{\prime }(t)-Y_{\nu }(t)J_{\nu }^{\prime }(t)=%
\dfrac{2}{\pi t}$ \ $(t>0)$.\newline
(iv) If $\nu \geq \frac{1}{2}$, then the function $t\longmapsto tN_{\nu }(t)$
is decreasing on $(0,\infty )$ and 
\begin{equation*}
\dfrac{2}{\pi t}\leq N_{\nu }(t)<\dfrac{2}{\pi }\dfrac{1}{\sqrt{t^{2}-\nu
^{2}}}\ \ \ \ (t>\nu ).
\end{equation*}%
If $0\leq \nu <\frac{1}{2}$, then the function $t\longmapsto tN_{\nu }(t)$
is increasing on $(0,\infty )$ and tends to $2/\pi $ as\ $t\rightarrow
\infty $.\newline
(v) The function $N_{\nu }$ is strictly decreasing on $(0,\infty )$.\newline
(vi) If $y(t)$ denotes $\sqrt{t}C_{\nu }(\kappa t)$, where $\kappa $ is a
non-zero constant, then 
\begin{equation*}
\frac{d^{2}y}{dt^{2}}+\left( \kappa ^{2}+\frac{\frac{1}{4}-\nu ^{2}}{t^{2}}%
\right) y=0\text{ \ \ \ }(t>0).
\end{equation*}
\end{lemma}

\begin{proof}
(i) and (ii) See pp.45, 66 of Watson \cite{Wat}.

(iii) See p.76, (1) of \cite{Wat}.

(iv) and (v) See Section 13.74 of \cite{Wat}.

(vi) See p.17, (1.8.9) of Szeg\"{o} \cite{Sze}.
\end{proof}

\bigskip

We now fix $b>1$ and define 
\begin{equation*}
U_{\nu }(\rho ,t)=J_{\nu }(\rho t)Y_{\nu }(\rho b)-J_{\nu }(\rho b)Y_{\nu
}(\rho t)\text{ \ \ \ }(\rho >0,t>0)\text{.}
\end{equation*}%
It is known \cite{Coc} (cf. Theorem X of Chapter VII in \cite{GMM}, and the
paragraph following the proof of Lemma \ref{uxy}\ below) that the zeros of
the function $\rho \mapsto U_{\nu }(\rho ,1)$ are all real and simple. We
denote by $(\rho _{\nu ,m})_{m\geq 1}$ the infinite sequence formed by the
positive zeros of this function arranged in increasing order. Clearly, 
\begin{equation}
\text{the function }x^{\prime }\mapsto U_{\nu }(\rho _{\nu ,m},\left\Vert
x^{\prime }\right\Vert )\text{ \ vanishes on }\partial A_{b}^{\prime }\text{
\ \ \ }(\nu \geq 0,m\geq 1).  \label{RE}
\end{equation}

Although the sequence $(\rho _{\nu ,m})_{m\geq 1}$ has been studied over
many years (as illustrated by \cite{McM}, \cite{SMM}), the following
important tool in the proof of Theorem \ref{main} appears to be new. We
defer its proof until Section \ref{appbegin}.

\begin{theorem}
\label{app}If $\nu \geq \frac{1}{2}$, then%
\begin{equation*}
\rho _{\nu ,m+1}-\rho _{\nu ,m}>\frac{\pi }{2b-1}\text{ \ \ \ }(m\geq 2).
\end{equation*}
\end{theorem}

Some further facts about $U_{\nu }$ and $(\rho _{\nu ,m})$ are assembled
below.

\begin{proposition}
\label{Unu}(i) If $\nu \geq \frac{1}{2}$, then 
\begin{equation}
\rho _{\nu ,m}>\dfrac{1}{b}\left( \nu +\dfrac{m}{4}\right) \text{ \ \ \ }%
(m\geq 1).  \label{lower}
\end{equation}%
Also, for any $\nu \geq 0$,%
\begin{equation}
\frac{\rho _{\nu ,m}}{m}\rightarrow \frac{\pi }{b-1}\text{ \ \ \ }%
(m\rightarrow \infty ).  \label{mcm}
\end{equation}%
\newline
(ii) If $U_{\nu }(\rho _{\nu ,m},a)=0$, where $0<a\leq b$, then%
\begin{equation*}
\left\{ \frac{\partial U_{\nu }}{\partial t}(\rho _{\nu ,m},a)\right\} ^{2}=%
\frac{4}{\pi ^{2}a^{2}}\frac{N_{\nu }(\rho _{\nu ,m}b)}{N_{\nu }(\rho _{\nu
,m}a)}\leq \frac{4}{\pi ^{2}a^{2}}.
\end{equation*}%
\newline
(iii) If $U_{\nu }(\rho _{\nu ,m},a)=0$, where $0<a<b$, then%
\begin{equation*}
\rho _{\nu ,m}^{2}\int_{a}^{b}\{U_{\nu }(\rho _{\nu ,m},t)\}^{2}~t~dt=\frac{2%
}{\pi ^{2}}\left( 1-\frac{N_{\nu }(\rho _{\nu ,m}b)}{N_{\nu }(\rho _{\nu
,m}a)}\right) \leq \frac{2}{\pi ^{2}}.
\end{equation*}%
\newline
(iv) If $\nu \geq \frac{1}{2}$, then%
\begin{equation}
\rho _{\nu ,m}^{2}\int_{1}^{b}\left\{ U_{\nu }(\rho _{\nu ,m},t)\right\}
^{2}t~dt\geq \frac{2}{\pi ^{2}}\frac{b-1}{b}\text{ \ \ \ }(m\geq 1).
\label{half}
\end{equation}%
Also,%
\begin{equation}
\rho _{0,m}^{2}\int_{1}^{b}\left\{ U_{0}(\rho _{0,m},t)\right\} ^{2}t~dt\geq 
\frac{2}{\pi ^{2}}\left( 1-\frac{N_{0}(\rho _{0,1}b)}{N_{0}(\rho _{0,1})}%
\right) \text{ \ \ \ }(m\geq 1).  \label{zero}
\end{equation}%
\newline
(v) If $a$ is the least positive zero of $U_{\nu }(\rho _{\nu ,m},\cdot )$,
where $\nu >0$, then%
\begin{equation*}
\left\vert U_{\nu }(\rho _{\nu ,m},t)\right\vert \leq \dfrac{t^{-\nu }}{\pi
\nu }\text{ \ \ \ }(0<t<a,m\geq 1).
\end{equation*}
\end{proposition}

\begin{proof}
(i) Let $j_{\nu ,1}^{\prime }$ denote the first positive zero of $J_{\nu
}^{\prime }$. Then $J_{\nu }$ is strictly increasing on $(0,j_{\nu
,1}^{\prime }]$, so $J_{\nu }(\rho )/J_{\nu }(\rho b)<1$ when $\rho \in (0,$ 
$j_{\nu ,1}^{\prime }/b]$, because $b>1$. Since $N_{\nu }$ is decreasing, by
Lemma \ref{JV}(v), we see that%
\begin{equation*}
\left\{ \frac{J_{\nu }(\rho )}{J_{\nu }(\rho b)}\right\} ^{2}<\frac{N_{\nu
}(\rho )}{N_{\nu }(\rho b)}\text{ \ \ \ }\left( 0<\rho \leq \frac{j_{\nu
,1}^{\prime }}{b}\right) ,
\end{equation*}%
whence%
\begin{equation*}
\left\{ J_{\nu }(\rho )Y_{\nu }(\rho b)\right\} ^{2}<\left\{ J_{\nu }(\rho
b)Y_{\nu }(\rho )\right\} ^{2}\text{ \ \ \ }\left( 0<\rho \leq \frac{j_{\nu
,1}^{\prime }}{b}\right) .
\end{equation*}%
It follows that $U_{\nu }(\cdot ,1)$ has no zeros on $(0,$ $j_{\nu
,1}^{\prime }/b]$, and so $\rho _{\nu ,1}>j_{\nu ,1}^{\prime }/b$. We know
from p.486, (3) of Watson \cite{Wat} that $j_{\nu ,1}^{\prime }>\sqrt{\nu
(\nu +2)}$, and clearly 
\begin{equation*}
\sqrt{\nu (\nu +2)}-\nu \geq (\sqrt{5}-1)/2>1/2\text{ \ \ \ }(\nu \geq 1/2).
\end{equation*}%
Thus $\rho _{\nu ,1}>(\nu +1/2)/b$, and so (\ref{lower}) holds when $m=1,2$.
The general case now follows from Theorem \ref{app}. The limit (\ref{mcm})
is contained in asymptotic estimates of McMahon \cite{McM} (cf. Cochran \cite%
{Coc}).

(ii) Let $C_{\nu }=\alpha J_{\nu }+\beta Y_{\nu }$ and $D_{\nu }=-\beta
J_{\nu }+\alpha Y_{\nu }$, where $\alpha ^{2}+\beta ^{2}\neq 0$. Then 
\begin{equation*}
\left\{ C_{\nu }\right\} ^{2}+\left\{ D_{\nu }\right\} ^{2}=(\alpha
^{2}+\beta ^{2})\left( \left\{ J_{\nu }\right\} ^{2}+\left\{ Y_{\nu
}\right\} ^{2}\right) =(\alpha ^{2}+\beta ^{2})N_{\nu },
\end{equation*}%
and%
\begin{equation*}
\left( C_{\nu }D_{\nu }^{\prime }-C_{\nu }^{\prime }D_{\nu }\right)
(t)=(\alpha ^{2}+\beta ^{2})(J_{\nu }Y_{\nu }^{\prime }-J_{\nu }^{\prime
}Y_{\nu })(t)=\frac{2(\alpha ^{2}+\beta ^{2})}{\pi t}
\end{equation*}%
by Lemma \ref{JV}(iii).\ If $\rho $ is a zero of $C_{\nu }$, we thus see
that $\rho C_{\nu }^{\prime }(\rho )D_{\nu }(\rho )=-2(\alpha ^{2}+\beta
^{2})/\pi $, and so%
\begin{equation*}
\left\{ \rho C_{\nu }^{\prime }(\rho )\right\} ^{2}=\frac{4}{\pi ^{2}}\frac{%
(\alpha ^{2}+\beta ^{2})^{2}}{\left\{ C_{\nu }(\rho )\right\} ^{2}+\left\{
D_{\nu }(\rho )\right\} ^{2}}=\frac{4}{\pi ^{2}}\frac{\alpha ^{2}+\beta ^{2}%
}{N_{\nu }(\rho )}.
\end{equation*}%
We will now apply this formula to the cylinder function $C_{\nu }(t)$
defined by $C_{\nu }(\rho _{\nu ,m}t)=U_{\nu }(\rho _{\nu ,m},t)$. Thus $%
\alpha =Y_{\nu }(\rho _{\nu ,m}b)$, $\beta =-J_{\nu }(\rho _{\nu ,m}b)$ and
so $\alpha ^{2}+\beta ^{2}=N_{\nu }(\rho _{\nu ,m}b)$. By putting $\rho
=\rho _{\nu ,m}a$, and noting that 
\begin{equation}
\frac{\partial U_{\nu }}{\partial t}(\rho _{\nu ,m},t)=\rho _{\nu ,m}C_{\nu
}^{\prime }(\rho _{\nu ,m}t),  \label{der}
\end{equation}%
we obtain the stated equality, and the subsequent inequality follows from
Lemma \ref{JV}(v).

(iii) We know from p.135, (11) of \cite{Wat} that 
\begin{equation*}
\int_{a}^{b}\{C_{\nu }(\rho t)\}^{2}~t~dt=\left[ \frac{t^{2}}{2}\left\{
\left( 1-\frac{\nu ^{2}}{\rho ^{2}t^{2}}\right) \{C_{\nu }(\rho
t)\}^{2}+\{C_{\nu }^{\prime }(\rho t)\}^{2}\right\} \right] _{a}^{b}
\end{equation*}%
for any cylinder function $C_{\nu }$. When $C_{\nu }(\rho _{\nu ,m}t)=U_{\nu
}(\rho _{\nu ,m},t)$, we can use (\ref{der}), and then part (ii), to see that%
\begin{eqnarray*}
\rho _{\nu ,m}^{2}\int_{a}^{b}\{U_{\nu }(\rho _{\nu ,m},t)\}^{2}~t~dt &=&%
\frac{b^{2}}{2}\left\{ \frac{\partial U_{\nu }}{\partial t}(\rho _{\nu
,m},b)\right\} ^{2}-\frac{a^{2}}{2}\left\{ \frac{\partial U_{\nu }}{\partial
t}(\rho _{\nu ,m},a)\right\} ^{2} \\
&=&\frac{2}{\pi ^{2}}\left( 1-\frac{N_{\nu }(\rho _{\nu ,m}b)}{N_{\nu }(\rho
_{\nu ,m}a)}\right) \leq \frac{2}{\pi ^{2}}.
\end{eqnarray*}

(iv) If $\nu \geq \frac{1}{2}$, then we know from Lemma \ref{JV}(iv) that $%
bN_{\nu }(\rho _{\nu ,m}b)\leq N_{\nu }(\rho _{\nu ,m})$, so (\ref{half})
follows from part (iii), with $a=1$. Next, we note from Section 4.1 of
Landau \cite{Lan} and the Nicholson integral formula for $N_{0}$ (see p.444,
(1) of \cite{Wat}) that the function $t\longmapsto -tN_{0}^{\prime
}(t)/N_{0}(t)$ is strictly increasing on $(0,\infty )$, whence 
\begin{equation*}
-\frac{btN_{0}^{\prime }(bt)}{N_{0}(bt)}>-\frac{tN_{0}^{\prime }(t)}{N_{0}(t)%
},\text{ \ or \ }bN_{0}(t)N_{0}^{\prime }(bt)-N_{0}(bt)N_{0}^{\prime }(t)<0%
\text{ \ \ \ }(t>0).
\end{equation*}%
It follows that the function $t\mapsto N_{0}(bt)/N_{0}(t)$ is decreasing, so%
\begin{equation*}
1-\frac{N_{0}(\rho _{0,m}b)}{N_{0}(\rho _{0,m})}\geq 1-\frac{N_{0}(\rho
_{0,1}b)}{N_{0}(\rho _{0,1})},
\end{equation*}%
and (\ref{zero}) now follows from part (iii).

(v) Let $y(t)=\sqrt{t}U_{\nu }(\rho _{\nu ,m},t)$, where $\nu >0$. Then%
\begin{equation*}
\frac{d}{dt}\left( t^{1-2\nu }\frac{d}{dt}\left( t^{\nu -1/2}y\right)
\right) =t^{1/2-\nu }y^{\prime \prime }-(\nu ^{2}-1/4)t^{-3/2-\nu }y=-\rho
_{\nu ,m}^{2}t^{1/2-\nu }y,
\end{equation*}%
by Lemma \ref{JV}(vi). Thus the left hand side of the above equation has the
opposite sign to $y$ on $(0,a)$. Let%
\begin{equation}
c=t^{1-2\nu }\frac{d}{dt}\left( t^{\nu -1/2}y\right) \left\vert
_{t=a}\right. =a^{1/2-\nu }y^{\prime }(a)=a^{1-\nu }\frac{\partial U_{\nu }}{%
\partial t}(\rho _{\nu ,m},a).  \label{c}
\end{equation}%
If $y<0$ on $(0,a)$, then $c>0$ and $t^{1-2\nu }\dfrac{d}{dt}\left( t^{\nu
-1/2}y\right) <c$ on $(0,a)$. These last two inequalities are reversed if $%
y>0$ on $(0,a)$. In either case, since $y(a)=0$, we see that%
\begin{equation*}
\left\vert t^{\nu -1/2}y(t)\right\vert \leq \left\vert c\right\vert
\int_{t}^{a}\tau ^{2\nu -1}d\tau \leq \frac{\left\vert c\right\vert }{2\nu }%
a^{2\nu }\text{ \ \ \ }(0<t<a),
\end{equation*}%
whence%
\begin{equation*}
\left\vert U_{\nu }(\rho _{\nu ,m},t)\right\vert =\left\vert
t^{-1/2}y(t)\right\vert \leq \frac{\left\vert c\right\vert }{2\nu }\frac{%
a^{2\nu }}{t^{\nu }}\leq \frac{a^{\nu }t^{-\nu }}{\pi \nu }\leq \frac{%
t^{-\nu }}{\pi \nu }\text{ \ \ \ }(0<t<a),
\end{equation*}%
by (\ref{c}), part (ii) and the fact that $a\leq 1$.
\end{proof}

\section{Some integrals and inequalities}

It will be convenient to define%
\begin{equation*}
\psi _{\nu }(t)=t^{\nu }-t^{-\nu }\text{ \ \ }(t>0,\nu >0).
\end{equation*}

\begin{proposition}
\label{coeff}Let $0<a<s<b$.\newline
(i) If $I_{\nu }(s)=\dint_{a}^{b}f_{\nu ,s}\left( t\right) C_{\nu }\left(
\rho t\right) t~dt$, where $\nu >0$, $\rho >0$ and 
\begin{equation*}
f_{\nu ,s}\left( t\right) =\left\{ 
\begin{array}{cc}
\dfrac{\psi _{\nu }(t/a)\psi _{\nu }(b/s)}{\psi _{\nu }(b/a)} & \left( a\leq
t\leq s\right) \\ 
\text{ } & \text{ } \\ 
\dfrac{\psi _{\nu }(s/a)\psi _{\nu }(b/t)}{\psi _{\nu }(b/a)} & \left(
s<t\leq b\right)%
\end{array}%
\right. ,
\end{equation*}%
then 
\begin{equation*}
\frac{\rho ^{2}}{2\nu }I_{\nu }(s)=C_{\nu }\left( \rho s\right) -\frac{%
C_{\nu }\left( \rho a\right) \psi _{\nu }(b/s)+C_{\nu }\left( \rho b\right)
\psi _{\nu }(s/a)}{\psi _{\nu }(b/a)}.
\end{equation*}%
(ii) If $I_{0}(s)=\dint_{a}^{b}f_{0,s}\left( t\right) C_{0}\left( \rho
t\right) t~dt$, where $\rho >0$ and 
\begin{equation*}
f_{0,s}\left( t\right) =\left\{ 
\begin{array}{cc}
\dfrac{\log (t/a)\log (b/s)}{\log (b/a)} & \left( a\leq t\leq s\right) \\ 
\text{ } & \text{ } \\ 
\dfrac{\log (s/a)\log (b/t)}{\log (b/a)} & \left( s<t\leq b\right)%
\end{array}%
\right. ,
\end{equation*}%
then 
\begin{equation*}
\rho ^{2}I_{0}(s)=C_{0}\left( \rho s\right) -\frac{C_{0}\left( \rho a\right)
\log (b/s)+C_{0}\left( \rho b\right) \log (s/a)}{\log (b/a)}.
\end{equation*}
\end{proposition}

\begin{proof}
(i) By Lemma \ref{JV}(i) 
\begin{eqnarray*}
\rho \int_{a}^{s}t^{\nu +1}C_{\nu }\left( \rho t\right) dt &=&\left[ t^{\nu
+1}C_{\nu +1}\left( \rho t\right) \right] _{a}^{s}, \\
-\rho \int_{a}^{s}t^{1-\nu }C_{\nu }\left( \rho t\right) dt &=&\left[
t^{1-\nu }C_{\nu -1}\left( \rho t\right) \right] _{a}^{s},
\end{eqnarray*}%
so 
\begin{equation*}
\rho \int_{a}^{s}t\psi _{\nu }(t/a)C_{\nu }\left( \rho t\right) dt=s^{\nu
+1}a^{-\nu }C_{\nu +1}\left( \rho s\right) +s^{1-\nu }a^{\nu }C_{\nu
-1}\left( \rho s\right) -\frac{2\nu }{\rho }C_{\nu }\left( \rho a\right) ,
\end{equation*}%
by Lemma \ref{JV}(ii). Similarly, 
\begin{eqnarray*}
\rho \int_{s}^{b}t\psi _{\nu }(b/t)C_{\nu }\left( \rho t\right) dt &=&\rho
b^{\nu }\int_{s}^{b}t^{1-\nu }C_{\nu }\left( \rho t\right) dt-\rho b^{-\nu
}\int_{s}^{b}t^{\nu +1}C_{\nu }\left( \rho t\right) dt \\
&=&-b^{\nu }\left[ t^{1-\nu }C_{\nu -1}\left( \rho t\right) \right]
_{s}^{b}-b^{-\nu }\left[ t^{\nu +1}C_{\nu +1}\left( \rho t\right) \right]
_{s}^{b} \\
&=&b^{\nu }s^{1-\nu }C_{\nu -1}\left( \rho s\right) +b^{-\nu }s^{\nu
+1}C_{\nu +1}\left( \rho s\right) -\frac{2\nu }{\rho }C_{\nu }\left( \rho
b\right) .
\end{eqnarray*}%
Hence%
\begin{eqnarray*}
\psi _{\nu }(b/a)\rho I_{\nu }(s) &=&\psi _{\nu }(b/s)\rho \int_{a}^{s}t\psi
_{\nu }(t/a)C_{\nu }\left( \rho t\right) dt+\psi _{\nu }(s/a)\rho
\int_{s}^{b}t\psi _{\nu }(b/t)C_{\nu }\left( \rho t\right) dt \\
&=&\psi _{\nu }(b/s)\left( s^{\nu +1}a^{-\nu }C_{\nu +1}\left( \rho s\right)
+s^{1-\nu }a^{\nu }C_{\nu -1}\left( \rho s\right) -\frac{2\nu }{\rho }C_{\nu
}\left( \rho a\right) \right) \\
&&+\psi _{\nu }(s/a)\left( b^{\nu }s^{1-\nu }C_{\nu -1}\left( \rho s\right)
+b^{-\nu }s^{\nu +1}C_{\nu +1}\left( \rho s\right) -\frac{2\nu }{\rho }%
C_{\nu }\left( \rho b\right) \right) .
\end{eqnarray*}%
The coefficients of the cylinder functions $C_{\nu +1},C_{\nu -1}$ in the
above expression are, respectively,%
\begin{eqnarray*}
a^{-\nu }\left( b^{\nu }s-b^{-\nu }s^{2\nu +1}\right) +b^{-\nu }\left(
a^{-\nu }s^{2\nu +1}-a^{\nu }s\right) &=&s\psi _{\nu }(b/a), \\
a^{\nu }\left( b^{\nu }s^{1-2\nu }-b^{-\nu }s\right) +b^{\nu }\left( a^{-\nu
}s-a^{\nu }s^{1-2\nu }\right) &=&s\psi _{\nu }(b/a).
\end{eqnarray*}%
Thus we can again use Lemma \ref{JV}(ii) to see that%
\begin{equation*}
\psi _{\nu }(b/a)\rho I_{\nu }(s)=s\psi _{\nu }(b/a)\frac{2\nu }{\rho s}%
C_{\nu }\left( \rho s\right) -\frac{2\nu }{\rho }C_{\nu }\left( \rho
a\right) \psi _{\nu }(b/s)-\frac{2\nu }{\rho }C_{\nu }\left( \rho b\right)
\psi _{\nu }(s/a),
\end{equation*}%
as claimed.

(ii) By Lemma \ref{JV}(i)%
\begin{eqnarray*}
\rho \int_{a}^{s}\log (t/a)C_{0}\left( \rho t\right) t~dt &=&\left[
tC_{1}\left( \rho t\right) \log (t/a)\right] _{a}^{s}-\int_{a}^{s}C_{1}%
\left( \rho t\right) dt \\
&=&sC_{1}\left( \rho s\right) \log (s/a)+\rho ^{-1}\left[ C_{0}\left( \rho
t\right) \right] _{a}^{s}, \\
\rho \int_{s}^{b}\log (b/t)C_{0}\left( \rho t\right) t~dt &=&\left[
tC_{1}\left( \rho t\right) \log (b/t)\right] _{s}^{b}+\int_{s}^{b}C_{1}%
\left( \rho t\right) dt \\
&=&-sC_{1}\left( \rho s\right) \log (b/s)-\rho ^{-1}\left[ C_{0}\left( \rho
t\right) \right] _{s}^{b}.
\end{eqnarray*}%
Hence%
\begin{eqnarray*}
\rho ^{2}\log (b/a)I_{0}(s) &=&\rho ^{2}\log (b/s)\int_{a}^{s}\log
(t/a)C_{0}\left( \rho t\right) t~dt+\rho ^{2}\log (s/a)\int_{s}^{b}\log
(b/t)C_{0}\left( \rho t\right) t~dt \\
&=&\log (b/s)\left( \rho sC_{1}\left( \rho s\right) \log (s/a)+C_{0}\left(
\rho s\right) -C_{0}\left( \rho a\right) \right) \\
&&+\log (s/a)\left( -\rho sC_{1}\left( \rho s\right) \log (b/s)-C_{0}\left(
\rho b\right) +C_{0}\left( \rho s\right) \right) \\
&=&\log (b/a)C_{0}\left( \rho s\right) -\log (b/s)C_{0}\left( \rho a\right)
-\log (s/a)C_{0}\left( \rho b\right) ,
\end{eqnarray*}%
as required.
\end{proof}

\begin{proposition}
\label{CS}If $0<a\leq s\leq b$ and $U_{\nu }(\rho _{\nu ,m},a)=0$,\ then 
\begin{equation}
\left\vert U_{\nu }(\rho _{\nu ,m},s)\right\vert \leq \frac{\rho _{\nu ,m}}{%
2\pi \nu }b\text{ \ \ }(\nu >0),\text{ \ \ \ }\left\vert U_{0}(\rho
_{0,m},s)\right\vert \leq \frac{\rho _{0,m}}{4\pi }b\log \frac{b}{a}
\label{CSa}
\end{equation}%
and%
\begin{equation}
\left\vert \frac{\partial U_{\nu }}{\partial t}(\rho _{\nu ,m},s)\right\vert
\leq \frac{\rho _{\nu ,m}}{\pi }\dfrac{b}{a}\text{ \ \ \ }(\nu \geq 0).
\label{CSb}
\end{equation}
\end{proposition}

\begin{proof}
We may assume that $s\in (a,b)$, since (\ref{CSa}) trivially holds when $%
s\in \{a,b\}$ and (\ref{CSb}) extends by continuity to the endpoints. Let 
\begin{equation*}
I_{\nu }(s)=\int_{a}^{b}f_{\nu ,s}\left( t\right) U_{\nu }(\rho _{\nu
,m},t)~t~dt,
\end{equation*}%
where $\nu \geq 0\ $and$\ f_{\nu ,s}$ is defined as in the previous
proposition. It is easy to see that 
\begin{equation*}
\max_{t\in \lbrack a,b]}f_{\nu ,s}(t)=f_{\nu ,s}(s)\text{ \ \ and \ \ }%
\max_{s\in \lbrack a,b]}f_{\nu ,s}(s)=f_{\nu ,\sqrt{ab}}(\sqrt{ab}).
\end{equation*}%
Further, 
\begin{equation*}
f_{\nu ,\sqrt{ab}}(\sqrt{ab})=\frac{\left\{ \psi _{\nu }(\sqrt{b/a})\right\}
^{2}}{\psi _{\nu }(b/a)}=\frac{1-(a/b)^{\nu }}{1+(a/b)^{\nu }}\leq 1\text{ \
\ \ }(\nu >0)
\end{equation*}%
and 
\begin{equation*}
f_{0,\sqrt{ab}}(\sqrt{ab})=\frac{\left\{ \log \sqrt{b/a}\right\} ^{2}}{\log
b/a}=\frac{\log (b/a)}{4}.
\end{equation*}%
Thus, by the Cauchy-Schwarz inequality and Proposition \ref{Unu}(iii), 
\begin{eqnarray}
\left\vert I_{\nu }(s)\right\vert &\leq &\left\{ \int_{a}^{b}\{f_{\nu
,s}(t)\}^{2}~t~dt\right\} ^{1/2}\left\{ \int_{a}^{b}\{U_{\nu }(\rho _{\nu
,m},t)\}^{2}~t~dt\right\} ^{1/2}  \notag \\
&\leq &\left\{ \frac{b^{2}-a^{2}}{2}\right\} ^{1/2}\left\{ \frac{2}{\pi
^{2}\rho _{\nu ,m}^{2}}\right\} ^{1/2}\leq \frac{b}{\pi \rho _{\nu ,m}}\text{
\ \ \ }(\nu >0),  \label{P1}
\end{eqnarray}%
and similarly%
\begin{equation}
\left\vert I_{0}(s)\right\vert \leq \frac{b\log (b/a)}{4\pi \rho _{0,m}}.
\label{P2}
\end{equation}

Next, we observe that%
\begin{equation}
I_{\nu }^{\prime }(s)=\int_{a}^{b}g_{\nu ,s}(t)U_{\nu }(\rho _{\nu
,m},t)~t~dt\text{ \ \ \ }(a<s<b),  \notag
\end{equation}%
where%
\begin{equation*}
g_{\nu ,s}\left( t\right) =\frac{d}{ds}f_{\nu ,s}(t)=\left\{ 
\begin{array}{cc}
-\dfrac{\nu }{s}\dfrac{\psi _{\nu }(t/a)}{\psi _{\nu }(b/a)}\left( \left( 
\dfrac{b}{s}\right) ^{\nu }+\left( \dfrac{b}{s}\right) ^{-\nu }\right) & 
\left( a\leq t<s\right) \\ 
\text{ } & \text{ } \\ 
\dfrac{\nu }{s}\dfrac{\psi _{\nu }(b/t)}{\psi _{\nu }(b/a)}\left( \left( 
\dfrac{s}{a}\right) ^{\nu }+\left( \dfrac{s}{a}\right) ^{-\nu }\right) & 
\left( s<t\leq b\right)%
\end{array}%
\right.
\end{equation*}%
when $\nu >0$, and%
\begin{equation*}
g_{0,s}\left( t\right) =\frac{d}{ds}f_{0,s}(t)=\left\{ 
\begin{array}{cc}
-\dfrac{1}{s}\dfrac{\log (t/a)}{\log (b/a)} & \left( a\leq t<s\right) \\ 
\text{ } & \text{ } \\ 
\dfrac{1}{s}\dfrac{\log (b/t)}{\log (b/a)} & \left( s<t\leq b\right)%
\end{array}%
\right. .
\end{equation*}%
Since%
\begin{equation*}
\psi _{\nu }(t/a)\leq \psi _{\nu }(s/a)\text{ \ \ }(a\leq t<s)\text{ \ \ and
\ \ }\psi _{\nu }(b/t)\leq \psi _{\nu }(b/s)\text{ \ \ \ }(s<t\leq b),
\end{equation*}%
and 
\begin{eqnarray*}
\psi _{\nu }\left( \frac{s}{a}\right) \left( \left( \dfrac{b}{s}\right)
^{\nu }+\left( \dfrac{b}{s}\right) ^{-\nu }\right) &=&\psi _{\nu }\left( 
\frac{b}{a}\right) +\psi _{\nu }\left( \frac{s^{2}}{ab}\right) \leq 2\psi
_{\nu }\left( \frac{b}{a}\right) , \\
\psi _{\nu }\left( \frac{b}{s}\right) \left( \left( \dfrac{s}{a}\right)
^{\nu }+\left( \dfrac{s}{a}\right) ^{-\nu }\right) &=&\psi _{\nu }\left( 
\frac{b}{a}\right) +\psi _{\nu }\left( \frac{ab}{s^{2}}\right) \leq 2\psi
_{\nu }\left( \frac{b}{a}\right) ,
\end{eqnarray*}%
we see that $\left\vert g_{\nu ,s}\left( t\right) \right\vert \leq 2\nu
/s\leq 2\nu /a$ when $\nu >0$. Thus, by the Cauchy-Schwarz inequality,%
\begin{equation}
\left\vert I_{\nu }^{\prime }(s)\right\vert \leq \frac{2\nu }{a}\left\{ 
\frac{b^{2}-a^{2}}{2}\right\} ^{1/2}\left\{ \frac{2}{\pi ^{2}\rho _{\nu
,m}^{2}}\right\} ^{1/2}\leq \frac{2\nu }{\pi \rho _{\nu ,m}}\dfrac{b}{a}%
\text{ \ \ \ }(\nu >0).  \label{P3}
\end{equation}%
Similarly, since clearly $\left\vert g_{0,s}\left( t\right) \right\vert \leq
1/s\leq 1/a$, we have%
\begin{equation}
\left\vert I_{0}^{\prime }(s)\right\vert \leq \frac{1}{a}\left\{ \frac{%
b^{2}-a^{2}}{2}\right\} ^{1/2}\left\{ \frac{2}{\pi ^{2}\rho _{0,m}^{2}}%
\right\} ^{1/2}\leq \frac{1}{\pi \rho _{0,m}}\dfrac{b}{a}.  \label{P4}
\end{equation}

The inequalities (\ref{CSa}) and (\ref{CSb}) follow from (\ref{P1}) - (\ref%
{P4}), since we can put $C_{\nu }(\rho t)=U_{\nu }(\rho ,t)$\ in Proposition %
\ref{coeff} to see that 
\begin{equation*}
U_{\nu }(\rho _{\nu ,m},s)=\dfrac{\rho _{\nu ,m}^{2}}{2\nu }I_{\nu }(s)\text{
\ \ }(\nu >0)\text{ \ \ and \ \ }U_{0}(\rho _{0,m},s)=\rho
_{0,m}^{2}I_{0}(s).
\end{equation*}
\end{proof}

\section{Intermediate series expansions}

We recall the following result from Section 1.11 of Titchmarsh \cite{Tit}.
(We have reformulated it using equation (1.6.4) there and Proposition \ref%
{Unu}(iii) above.)

\begin{proposition}
\label{SL}Let $f:$ $[1,b]\rightarrow \mathbb{R}$ be a continuous function of
bounded variation and let 
\begin{equation*}
a_{m}=\frac{1}{\int_{1}^{b}\left\{ U_{\nu }\left( \rho _{\nu ,m},\tau
\right) \right\} ^{2}\tau ~d\tau }\int_{1}^{b}f\left( \tau \right) U_{\nu
}\left( \rho _{\nu ,m},\tau \right) \tau ~d\tau .
\end{equation*}%
Then the series $\sum_{m=1}^{\infty }a_{m}U_{\nu }\left( \rho _{\nu
,m},t\right) $ converges pointwise to $f\left( t\right) $ on $(1,b)$.
\end{proposition}

Formula (\ref{Af}) below is stated without proof by Carslaw \cite{Car}.

\begin{proposition}
\label{Uexp}Let $1<s<b$.\newline
(a)\ If $\nu >0$, then 
\begin{equation}
2\nu \sum\limits_{m=1}^{\infty }\frac{U_{\nu }(\rho _{\nu ,m},s)U_{\nu
}(\rho _{\nu ,m},t)}{\rho _{\nu ,m}^{2}\int_{1}^{b}\left\{ U_{\nu }\left(
\rho _{\nu ,m},\tau \right) \right\} ^{2}\tau ~d\tau }=\left\{ 
\begin{array}{cc}
\dfrac{\psi _{\nu }(t)\psi _{\nu }(b/s)}{\psi _{\nu }(b)} & \left( 1\leq
t\leq s\right) \\ 
\text{ } & \text{ } \\ 
\dfrac{\psi _{\nu }(s)\psi _{\nu }(b/t)}{\psi _{\nu }(b)} & \left( s<t\leq
b\right)%
\end{array}%
\right. ,  \label{Af}
\end{equation}%
and the series converges uniformly for $t\in \lbrack 1,b]$.\newline
(b) In the case where $\nu =0$,%
\begin{equation}
\sum\limits_{m=1}^{\infty }\frac{U_{0}(\rho _{0,m},s)U_{0}(\rho _{0,m},t)}{%
\rho _{0,m}^{2}\int_{1}^{b}\left\{ U_{0}\left( \rho _{0,m},\tau \right)
\right\} ^{2}\tau ~d\tau }=\left\{ 
\begin{array}{cc}
\dfrac{(\log t)\log (b/s)}{\log b} & (1<t\leq s) \\ 
\text{ } & \text{ } \\ 
\dfrac{(\log s)\log (b/t)}{\log b} & (s<t<b)%
\end{array}%
\right. ,  \label{Bf}
\end{equation}%
and the series converges uniformly for $t\in \lbrack 1,b]$.
\end{proposition}

\begin{proof}
We know from p.199 of \cite{Wat} that any cylinder function $C_{\nu }$
satisfies $C_{\nu }(t)=O(t^{-1/2})$ as $t\rightarrow \infty $. Applying this
estimate separately to each factor in the definition of $U_{\nu }(\rho _{\nu
,m},t)$, we see that $\left\vert U_{\nu }(\rho _{\nu ,m},\cdot )\right\vert
\leq C(b,\nu )/\rho _{\nu ,m}$ on $[1,b]$. Thus, by parts (i) and (iv) of
Proposition \ref{Unu}, the series in (\ref{Af}) converges uniformly on $%
[1,b] $. Part (a) now follows from Proposition \ref{SL} and the fact that 
\begin{equation*}
\int_{1}^{b}f_{\nu ,s}\left( \tau \right) U_{\nu }\left( \rho _{\nu ,m},\tau
\right) \tau ~d\tau =\frac{2\nu }{\rho _{\nu ,m}^{2}}U_{\nu }\left( \rho
_{\nu ,m},s\right) ,
\end{equation*}%
by Proposition \ref{coeff}(i),$\ $where $f_{\nu ,s}\left( t\right) $ denotes
the right hand side of (\ref{Af}).

Part (b) follows in similar fashion from Proposition \ref{coeff}(ii)\textbf{.%
}
\end{proof}

\bigskip

If $\lambda >0$, let $P_{n}^{(\lambda )}$ be the usual ultraspherical
(Gegenbauer) polynomial defined by the expansion 
\begin{equation*}
(1-2tu+u^{2})^{-\lambda }=\sum\limits_{n=0}^{\infty }P_{n}^{(\lambda
)}(t)u^{n}\text{ \ \ \ }(\left\vert t\right\vert \leq 1,\left\vert
u\right\vert <1).
\end{equation*}%
(See Section 4.7 of Szeg\"{o} \cite{Sze}, or Chapter IV of Stein and Weiss 
\cite{SW}.) We note for future reference that%
\begin{equation}
\left\vert P_{n}^{(\lambda )}(t)\right\vert \leq P_{n}^{(\lambda )}(1)=%
\binom{n+2\lambda -1}{n}\text{ \ \ \ }(\left\vert t\right\vert \leq 1)
\label{pnl}
\end{equation}%
(see Lemma 6(i) of \cite{GR1}). Also, let $T_{n}(t)$ be the Chebychev
polynomial given by $\cos (n\cos ^{-1}t)$ when $\left\vert t\right\vert \leq
1$, and let 
\begin{equation*}
\nu _{n}=n+\frac{N-3}{2}\text{ \ \ \ }(n\geq 0).
\end{equation*}

We will need the following known expansions for the Green function $%
G_{A_{b}^{\prime }}(\cdot ,\cdot )$ of the annular region $A_{b}^{\prime }$
in $\mathbb{R}^{N-1}$ $(N\geq 3)$.

\begin{proposition}
\label{B}Suppose that $y^{\prime }\in A_{b}^{\prime }$.\newline
(i) Let $N\geq 4$. If $1<\left\Vert x^{\prime }\right\Vert <\left\Vert
y^{\prime }\right\Vert $, then 
\begin{equation}
G_{A_{b}^{\prime }}(x^{\prime },y^{\prime })=\left( \left\Vert x^{\prime
}\right\Vert \left\Vert y^{\prime }\right\Vert \right) ^{\frac{3-N}{2}%
}\sum_{n=0}^{\infty }P_{n}^{\left( \frac{N-3}{2}\right) }\left( \dfrac{%
\left\langle x^{\prime },y^{\prime }\right\rangle }{\left\Vert x^{\prime
}\right\Vert \left\Vert y^{\prime }\right\Vert }\right) \frac{\psi _{\nu
_{n}}(\left\Vert x^{\prime }\right\Vert )\psi _{\nu _{n}}\left( b/\left\Vert
y^{\prime }\right\Vert \right) }{\psi _{\nu _{n}}(b)};  \label{R1}
\end{equation}%
and, if $\left\Vert y^{\prime }\right\Vert <\left\Vert x^{\prime
}\right\Vert <b$,\ then%
\begin{equation}
G_{A_{b}^{\prime }}(x^{\prime },y^{\prime })=\left( \left\Vert x^{\prime
}\right\Vert \left\Vert y^{\prime }\right\Vert \right) ^{\frac{3-N}{2}%
}\sum_{n=0}^{\infty }P_{n}^{\left( \frac{N-3}{2}\right) }\left( \dfrac{%
\left\langle x^{\prime },y^{\prime }\right\rangle }{\left\Vert x^{\prime
}\right\Vert \left\Vert y^{\prime }\right\Vert }\right) \frac{\psi _{\nu
_{n}}(\left\Vert y^{\prime }\right\Vert )\psi _{\nu _{n}}\left( b/\left\Vert
x^{\prime }\right\Vert \right) }{\psi _{\nu _{n}}(b)}.  \label{S1}
\end{equation}%
(ii)\ Let $N=3$. If $1<\left\Vert x^{\prime }\right\Vert <\left\Vert
y^{\prime }\right\Vert $, then%
\begin{equation}
G_{A_{b}^{\prime }}(x^{\prime },y^{\prime })=\frac{\log (b/\left\Vert
y^{\prime }\right\Vert )}{\log b}\log \left\Vert x^{\prime }\right\Vert
+\sum_{n=1}^{\infty }\frac{1}{n}T_{n}\left( \frac{\left\langle x^{\prime
},y^{\prime }\right\rangle }{\left\Vert x^{\prime }\right\Vert \left\Vert
y^{\prime }\right\Vert }\right) \frac{\psi _{n}(\left\Vert x^{\prime
}\right\Vert )\psi _{n}\left( b/\left\Vert y^{\prime }\right\Vert \right) }{%
\psi _{n}(b)};  \label{R2}
\end{equation}%
and, if $\left\Vert y^{\prime }\right\Vert <\left\Vert x^{\prime
}\right\Vert <b$,\ then%
\begin{equation}
G_{A_{b}^{\prime }}(x^{\prime },y^{\prime })=\frac{\log (b/\left\Vert
x^{\prime }\right\Vert )}{\log b}\log \left\Vert y^{\prime }\right\Vert
+\sum_{n=1}^{\infty }\frac{1}{n}T_{n}\left( \frac{\left\langle x^{\prime
},y^{\prime }\right\rangle }{\left\Vert x^{\prime }\right\Vert \left\Vert
y^{\prime }\right\Vert }\right) \frac{\psi _{n}(\left\Vert y^{\prime
}\right\Vert )\psi _{n}\left( b/\left\Vert x^{\prime }\right\Vert \right) }{%
\psi _{n}(b)}.  \label{S2}
\end{equation}
\end{proposition}

\begin{proof}
(i) This follows by dilation from Corollary 1.1 of Grossi and Vujadinovic 
\cite{GV}.

(ii) This follows by combining Proposition 2.1 of Grossi and Takahashi \cite%
{GT} (cf. Hickey \cite{Hic}) with the expansions%
\begin{equation*}
-\log \left\Vert \frac{x^{\prime }}{\left\Vert y^{\prime }\right\Vert }-%
\frac{y^{\prime }}{\left\Vert y^{\prime }\right\Vert }\right\Vert
=\sum\limits_{n=1}^{\infty }\frac{1}{n}T_{n}\left( \frac{\left\langle
x^{\prime },y^{\prime }\right\rangle }{\left\Vert x^{\prime }\right\Vert
\left\Vert y^{\prime }\right\Vert }\right) \left( \frac{\left\Vert x^{\prime
}\right\Vert }{\left\Vert y^{\prime }\right\Vert }\right) ^{n}\text{ \ \ \ }%
(\left\Vert x^{\prime }\right\Vert <\left\Vert y^{\prime }\right\Vert ),
\end{equation*}%
\begin{equation*}
-\log \left\Vert \frac{x^{\prime }}{\left\Vert x^{\prime }\right\Vert }-%
\frac{y^{\prime }}{\left\Vert x^{\prime }\right\Vert }\right\Vert
=\sum\limits_{n=1}^{\infty }\frac{1}{n}T_{n}\left( \frac{\left\langle
x^{\prime },y^{\prime }\right\rangle }{\left\Vert x^{\prime }\right\Vert
\left\Vert y^{\prime }\right\Vert }\right) \left( \frac{\left\Vert y^{\prime
}\right\Vert }{\left\Vert x^{\prime }\right\Vert }\right) ^{n}\text{ \ \ \ }%
(\left\Vert y^{\prime }\right\Vert <\left\Vert x^{\prime }\right\Vert ).
\end{equation*}
\end{proof}

\bigskip

Let $y^{\prime }\in A_{b}^{\prime }$ and $\delta \in (0,1)$, and let $%
S_{y^{\prime }}$ be the sphere in $\mathbb{R}^{N-1}$\ centred at $0^{\prime
} $ that contains $y^{\prime }$. We define $\mu _{y^{\prime },\delta }$ to
be the probability measure on $S_{y^{\prime }}$ that has density with
respect to surface area measure proportional to%
\begin{equation*}
\exp \left( -\left( 1-\frac{\left\Vert z^{\prime }-y^{\prime }\right\Vert
^{2}}{\delta ^{2}\left\Vert y^{\prime }\right\Vert ^{2}}\right) ^{-1}\right) 
\text{ \ when }\left\Vert z^{\prime }-y^{\prime }\right\Vert <\delta
\left\Vert y^{\prime }\right\Vert \text{, \ \ and }0\text{ otherwise.}
\end{equation*}%
We further define the Green potential 
\begin{equation*}
G_{A_{b}^{\prime }}\mu _{y^{\prime },\delta }(x^{\prime })=\int
G_{A_{b}^{\prime }}(x^{\prime },z^{\prime })~d\mu _{y^{\prime },\delta
}(z^{\prime })\text{ \ \ \ }(x^{\prime }\in A_{b}^{\prime }),
\end{equation*}%
and the function%
\begin{equation*}
P_{n}^{(\frac{N-3}{2})}\mu _{y^{\prime },\delta }(x^{\prime })=\int
P_{n}^{\left( \frac{N-3}{2}\right) }\left( \dfrac{\left\langle x^{\prime
},z^{\prime }\right\rangle }{\left\Vert x^{\prime }\right\Vert \left\Vert
z^{\prime }\right\Vert }\right) ~d\mu _{y^{\prime },\delta }(z^{\prime })%
\text{ \ \ \ }(x^{\prime }\in \mathbb{R}^{N-1}\backslash \{0^{\prime }\})
\end{equation*}%
when $N\geq 4$. When $N=3$\ the function $T_{n}\mu _{y^{\prime },\delta }$
is defined from $T_{n}$ analogously.

We recall the following result (see \cite{Kal}).

\begin{proposition}
\label{F}Let $f\in C^{\infty }(\partial B^{\prime })$ and let $c_{i,j}$ be
the Fourier coefficients of $f$ with respect to an orthonormal basis $%
\{H_{i,j}:j=1,...,M(i)\}$ of the spherical harmonics of degree $i$ in $%
\mathbb{R}^{N-1}$. Then the series $\sum_{i=0}^{\infty
}\sum_{j=1}^{M(i)}c_{i,j}H_{i,j}$ converges uniformly on $\partial B^{\prime
}$ to $f$, and so the series 
\begin{equation*}
\sum_{i=0}^{\infty }\left\Vert x^{\prime }\right\Vert
^{i}\sum_{j=1}^{M(i)}c_{i,j}H_{i,j}\left( \frac{x^{\prime }}{\left\Vert
x^{\prime }\right\Vert }\right)
\end{equation*}%
converges uniformly on $B^{\prime }\backslash \{0^{\prime }\}$ to the
Poisson integral of $f$ in $B^{\prime }$.
\end{proposition}

\begin{remark}
\label{BR}By Proposition \ref{B} we obtain formulae for $G_{A_{b}^{\prime
}}\mu _{y^{\prime },\delta }(x^{\prime })$ if we replace $P_{n}^{\left( 
\frac{N-3}{2}\right) }\left( \frac{\left\langle x^{\prime },y^{\prime
}\right\rangle }{\left\Vert x^{\prime }\right\Vert \left\Vert y^{\prime
}\right\Vert }\right) $ by $P_{n}^{(\frac{N-3}{2})}\mu _{y^{\prime },\delta
}(x^{\prime })$ in (\ref{R1}) and (\ref{S1}), and $T_{n}\left( \frac{%
\left\langle x^{\prime },y^{\prime }\right\rangle }{\left\Vert x^{\prime
}\right\Vert \left\Vert y^{\prime }\right\Vert }\right) $ by $T_{n}\mu
_{y^{\prime },\delta }(x^{\prime })$ in (\ref{R2}) and (\ref{S2}). Further,
the series in (\ref{R1}) and (\ref{R2}) would now converge uniformly on $%
\{x^{\prime }:1<\left\Vert x^{\prime }\right\Vert \leq \left\Vert y^{\prime
}\right\Vert \}$. When $N\geq 4$ this follows from the proof of Corollary
1.1 in \cite{GV} with the additional ingredient that the restriction of the
Newtonian potential%
\begin{equation*}
x^{\prime }\longmapsto \int \left\Vert x^{\prime }-z^{\prime }\right\Vert
^{3-N}d\mu _{y^{\prime },\delta }(z^{\prime })\text{ \ \ \ }(x^{\prime }\in 
\mathbb{R}^{N-1})
\end{equation*}%
to $S_{y^{\prime }}$ is $C^{\infty }$ (cf. Theorem 3.3.3 of \cite{AG}), and
so we can appeal to the preceding proposition. The case where $N=3$ follows
similarly from \cite{GT}. Further, inversion can be used to show that the
series in (\ref{S1}) and (\ref{S2}) would converge uniformly on $\{x^{\prime
}:\left\Vert y^{\prime }\right\Vert \leq \left\Vert x^{\prime }\right\Vert
<b\}$.
\end{remark}

\section{Proofs of main results\label{endmain}}

Let $\eta _{t}$ denote the unit measure concentrated at $t\in \mathbb{R}$.

\begin{lemma}
\label{hc}For any $n\geq 0,m\geq 1$ and $y\in \Omega _{b}$, let $u_{n,m,y}$
be the function on $(\mathbb{R}^{N-1}\backslash \{0^{\prime }\})\times 
\mathbb{R}$\ defined by%
\begin{equation*}
x\mapsto \left\Vert x^{\prime }\right\Vert ^{\frac{3-N}{2}}P_{n}^{\left( 
\frac{N-3}{2}\right) }\left( \dfrac{\left\langle x^{\prime },y^{\prime
}\right\rangle }{\left\Vert x^{\prime }\right\Vert \left\Vert y^{\prime
}\right\Vert }\right) \frac{U_{\nu _{n}}(\rho _{\nu _{n},m},\left\Vert
x^{\prime }\right\Vert )U_{\nu _{n}}(\rho _{\nu _{n},m},\left\Vert y^{\prime
}\right\Vert )}{\rho _{\nu _{n},m}\int_{1}^{b}U_{\nu _{n}}^{2}(\rho _{\nu
_{n},m},t)~t~dt}e^{-\rho _{\nu _{n},m}\left\vert x_{N}-y_{N}\right\vert }%
\text{ \ \ \ }(N\geq 4),
\end{equation*}%
\begin{equation*}
x\mapsto T_{n}\left( \frac{\left\langle x^{\prime },y^{\prime }\right\rangle 
}{\left\Vert x^{\prime }\right\Vert \left\Vert y^{\prime }\right\Vert }%
\right) \frac{U_{\nu _{n}}(\rho _{\nu _{n},m},\left\Vert x^{\prime
}\right\Vert )U_{\nu _{n}}(\rho _{\nu _{n},m},\left\Vert y^{\prime
}\right\Vert )}{\rho _{\nu _{n},m}\int_{1}^{b}U_{\nu _{n}}^{2}(\rho _{\nu
_{n},m},t)~t~dt}e^{-\rho _{\nu _{n},m}\left\vert x_{N}-y_{N}\right\vert }%
\text{ \ \ \ }(N=3).
\end{equation*}%
Then $u_{n,m,y}$\newline
(i) is harmonic on $\left( \mathbb{R}^{N-1}\backslash \{0^{\prime }\}\right)
\times \left( \mathbb{R}\backslash \{y_{N}\}\right) $;\newline
(ii) continuously vanishes on $\partial A_{b}^{\prime }\times \mathbb{R}$;%
\newline
(iii) has distributional Laplacian on $\left( \mathbb{R}^{N-1}\backslash
\{0^{\prime }\}\right) \times \mathbb{R}$ given by 
\begin{equation*}
-2\rho _{\nu _{n},m}u_{n,m,y}(z^{\prime },y_{N})dz^{\prime }d\eta _{y_{N}}.
\end{equation*}
\end{lemma}

\begin{proof}
Parts (i) and (iii) are proved by the same arguments as were used to
establish parts (i) and (iv) of Lemma 11 in \cite{GR1}. Part (ii) follows
from (\ref{RE}).
\end{proof}

\bigskip

The binomial coefficient $\binom{n+N-4}{n}$, which appears in several
estimates below, should be interpreted as $1$ when $N=3$. We denote the
distance from $y^{\prime }$ to $\mathbb{R}^{N-1}\backslash A_{b}^{\prime }$
by%
\begin{equation*}
d(y^{\prime })=\min \{\left\Vert y^{\prime }\right\Vert -1,b-\left\Vert
y^{\prime }\right\Vert \}\text{ \ \ \ }(y^{\prime }\in \overline{%
A_{b}^{\prime }}).
\end{equation*}%
Also, we will write $C(\alpha ,\beta ,...)$ for a positive constant
depending at most on $\alpha ,\beta ,...$, not necessarily the same on any
two occurrences.

\begin{lemma}
\label{est}Let $n\geq 0,m\geq 1,y\in \Omega _{b}$, and let $u_{n,m,y}$ be as
in Lemma \ref{hc}. Then\newline
(i) $\left\vert u_{n,m,y}(x)\right\vert \leq C(b)\dbinom{n+N-4}{n}\rho _{\nu
_{n},m}^{3}d(y^{\prime })e^{-\rho _{\nu _{n},m}\left\vert
x_{N}-y_{N}\right\vert }$ \ $(1\leq \left\Vert x^{\prime }\right\Vert \leq
b) $;\newline
(ii) $\left\vert u_{n,m,y}(x)\right\vert \leq C(b)\dbinom{n+N-4}{n}\dfrac{%
\rho _{\nu _{n},m}^{2}}{m}d(y^{\prime })e^{-\rho _{\nu _{n},m}\left\vert
x_{N}-y_{N}\right\vert }$ \ $(b<\left\Vert x^{\prime }\right\Vert )$;\newline
(iii) $\left\vert u_{n,m,y}(x)\right\vert \leq C(b)\dbinom{n+N-4}{n}\dfrac{%
\rho _{\nu _{n},m}^{3}}{\left\Vert x^{\prime }\right\Vert ^{\nu _{n}+(N-1)/2}%
}d(y^{\prime })e^{-\rho _{\nu _{n},m}\left\vert x_{N}-y_{N}\right\vert }$ \ $%
(0<\left\Vert x^{\prime }\right\Vert <1)$.
\end{lemma}

\begin{proof}
Since either $\nu _{n}\geq 
{\frac12}%
$ or $\nu _{n}=0$, we see from Proposition \ref{Unu}(iv) that 
\begin{equation*}
\frac{1}{\rho _{\nu _{n},m}\int_{1}^{b}\left\{ U_{\nu _{n}}(\rho _{\nu
_{n},m},t)\right\} ^{2}t~dt}\leq C(b)\rho _{\nu _{n},m}.
\end{equation*}%
Further, by (\ref{CSb}), (\ref{RE}) and the mean value theorem,%
\begin{equation*}
\left\vert U_{\nu _{n}}(\rho _{\nu _{n},m},\left\Vert y^{\prime }\right\Vert
)\right\vert \leq \frac{\rho _{\nu _{n},m}b}{\pi }d(y^{\prime }).
\end{equation*}%
In view of (\ref{pnl}) it only remains to establish appropriate estimates
for $U_{\nu _{n}}(\rho _{\nu _{n},m},\left\Vert x^{\prime }\right\Vert )$ in
each of the three stated regions.

Proposition \ref{CS} (with $a=1$) shows that%
\begin{equation*}
\left\vert U_{\nu _{n}}(\rho _{\nu _{n},m},\left\Vert x^{\prime }\right\Vert
)\right\vert \leq C(b)\rho _{\nu _{n},m}\text{ \ \ \ }(1\leq \left\Vert
x^{\prime }\right\Vert \leq b),
\end{equation*}%
so part (i) is established.

When $\left\Vert x^{\prime }\right\Vert >b$, we use the arithmetic-geometric
means inequality and then Lemma \ref{JV}(v) to see that 
\begin{equation*}
\left\vert U_{\nu _{n}}(\rho _{\nu _{n},m},\left\Vert x^{\prime }\right\Vert
)\right\vert \leq \frac{N_{\nu _{n}}(\rho _{\nu _{n},m}\left\Vert x^{\prime
}\right\Vert )+N_{\nu _{n}}(\rho _{\nu _{n},m}b)}{2}\leq N_{\nu _{n}}(\rho
_{\nu _{n},m}b),
\end{equation*}%
and hence that%
\begin{equation*}
\left\vert U_{\nu _{n}}(\rho _{\nu _{n},m},\left\Vert x^{\prime }\right\Vert
)\right\vert \leq \frac{C}{m},
\end{equation*}%
by Proposition \ref{Unu}(i) and Lemma \ref{JV}(iv). Part (ii) now follows.

To prove part (iii), let $a$ denote the least positive zero of $U_{\nu
_{n}}(\rho _{\nu _{n},m},\cdot )$. Thus $a\in (0,1]$. If $\nu _{n}\geq 
{\frac12}%
$, then we see from Proposition \ref{CS} that 
\begin{equation*}
\left\vert U_{\nu _{n}}(\rho _{\nu _{n},m},\left\Vert x^{\prime }\right\Vert
)\right\vert \leq \frac{\rho _{\nu _{n},m}}{\pi }b\text{ \ \ \ }(a\leq
\left\Vert x^{\prime }\right\Vert \leq 1),
\end{equation*}%
and from Proposition \ref{Unu}(v) that%
\begin{equation*}
\left\vert U_{\nu _{n}}(\rho _{\nu _{n},m},\left\Vert x^{\prime }\right\Vert
)\right\vert \leq \frac{2\left\Vert x^{\prime }\right\Vert ^{-\nu _{n}}}{\pi 
}\text{ \ \ \ }(0<\left\Vert x^{\prime }\right\Vert <a),
\end{equation*}%
whence%
\begin{equation*}
\left\vert U_{\nu _{n}}(\rho _{\nu _{n},m},\left\Vert x^{\prime }\right\Vert
)\right\vert \leq C(b)\rho _{\nu _{n},m}\left\Vert x^{\prime }\right\Vert
^{-\nu _{n}}\text{ \ \ \ }(0<\left\Vert x^{\prime }\right\Vert <1).
\end{equation*}%
If $\nu _{n}=0$, we instead observe that 
\begin{eqnarray*}
\left\vert U_{0}(\rho _{0,m},\left\Vert x^{\prime }\right\Vert )\right\vert
&\leq &\frac{N_{0}(\rho _{0,m}\left\Vert x^{\prime }\right\Vert )+N_{0}(\rho
_{0,m}b)}{2}\leq N_{0}(\rho _{0,m}\left\Vert x^{\prime }\right\Vert ) \\
&\leq &\frac{2}{\pi \rho _{0,m}\left\Vert x^{\prime }\right\Vert }\text{ \ \
\ }(0<\left\Vert x^{\prime }\right\Vert <1),
\end{eqnarray*}%
by Lemma \ref{JV}(iv).
\end{proof}

\bigskip

Let $n\geq 0$, $\delta \in (0,1)$ and $y^{\prime }\in A_{b}^{\prime }$. We
define%
\begin{equation*}
h_{n,y^{\prime }}^{\delta }(x^{\prime })=\left\{ 
\begin{array}{cc}
\dfrac{\left\Vert x^{\prime }\right\Vert ^{\frac{3-N}{2}}}{\nu _{n}}%
P_{n}^{\left( \frac{N-3}{2}\right) }\mu _{y^{\prime },\delta }\left(
x^{\prime }\right) \dfrac{\psi _{\nu _{n}}(\left\Vert x^{\prime }\right\Vert
)\psi _{\nu _{n}}\left( b/\left\Vert y^{\prime }\right\Vert \right) }{\psi
_{\nu _{n}}(b)} & (1<\left\Vert x^{\prime }\right\Vert \leq \left\Vert
y^{\prime }\right\Vert ) \\ 
\text{ } & \text{ } \\ 
\dfrac{\left\Vert x^{\prime }\right\Vert ^{\frac{3-N}{2}}}{\nu _{n}}%
P_{n}^{\left( \frac{N-3}{2}\right) }\mu _{y^{\prime },\delta }\left(
x^{\prime }\right) \dfrac{\psi _{\nu _{n}}(\left\Vert y^{\prime }\right\Vert
)\psi _{\nu _{n}}\left( b/\left\Vert x^{\prime }\right\Vert \right) }{\psi
_{\nu _{n}}(b)} & (\left\Vert y^{\prime }\right\Vert <\left\Vert x^{\prime
}\right\Vert <b)%
\end{array}%
\right.
\end{equation*}%
if $N\geq 4$, 
\begin{equation*}
h_{n,y^{\prime }}^{\delta }(x^{\prime })=\left\{ 
\begin{array}{cc}
\dfrac{1}{n}T_{n}\mu _{y^{\prime },\delta }\left( x^{\prime }\right) \dfrac{%
\psi _{n}(\left\Vert x^{\prime }\right\Vert )\psi _{n}\left( b/\left\Vert
y^{\prime }\right\Vert \right) }{\psi _{n}(b)} & (1<\left\Vert x^{\prime
}\right\Vert \leq \left\Vert y^{\prime }\right\Vert ) \\ 
\text{ } & \text{ } \\ 
\dfrac{1}{n}T_{n\mu _{y^{\prime },\delta }}\left( x^{\prime }\right) \dfrac{%
\psi _{n}(\left\Vert y^{\prime }\right\Vert )\psi _{n}\left( b/\left\Vert
x^{\prime }\right\Vert \right) }{\psi _{n}(b)} & (\left\Vert y^{\prime
}\right\Vert <\left\Vert x^{\prime }\right\Vert <b)%
\end{array}%
\right.
\end{equation*}%
if $N=3$ and $n\geq 1$, and when $N=3$ and $n=0$ we write%
\begin{equation*}
h_{0,y^{\prime }}^{\delta }(x^{\prime })=\left\{ 
\begin{array}{cc}
2\dfrac{\log \left( b/\left\Vert y^{\prime }\right\Vert \right) }{\log b}%
\log \left\Vert x^{\prime }\right\Vert & (1<\left\Vert x^{\prime
}\right\Vert \leq \left\Vert y^{\prime }\right\Vert ) \\ 
\text{ } & \text{ } \\ 
2\dfrac{\log \left( b/\left\Vert x^{\prime }\right\Vert \right) }{\log b}%
\log \left\Vert y^{\prime }\right\Vert & (\left\Vert y^{\prime }\right\Vert
<\left\Vert x^{\prime }\right\Vert <b)%
\end{array}%
\right. .
\end{equation*}%
Further, let $u_{n,m,y}^{\delta }$ have the same definition as $u_{n,m,y}$,
except that we use $P_{n}^{(\frac{N-3}{2})}\mu _{y^{\prime },\delta
}(x^{\prime })$ and $T_{n}\mu _{y^{\prime },\delta }(x^{\prime })$ in place
of $P_{n}^{\left( \frac{N-3}{2}\right) }\left( \frac{\left\langle x^{\prime
},y^{\prime }\right\rangle }{\left\Vert x^{\prime }\right\Vert \left\Vert
y^{\prime }\right\Vert }\right) $ and $T_{n}\left( \frac{\left\langle
x^{\prime },y^{\prime }\right\rangle }{\left\Vert x^{\prime }\right\Vert
\left\Vert y^{\prime }\right\Vert }\right) $, respectively.

\begin{remark}
\label{del}Lemmas \ref{hc} and \ref{est}\ clearly remain true if we replace $%
u_{n,m,y}$ by $u_{n,m,y}^{\delta }$ throughout.
\end{remark}

We define 
\begin{equation*}
a_{N}=\sigma _{N}(N-2)\text{ \ when \ }N\geq 3,\text{ \ and }a_{2}=\sigma
_{2},
\end{equation*}%
where $\sigma _{N}$ denotes the surface area of the unit sphere in $\mathbb{R%
}^{N}$.

\begin{lemma}
\label{La}Let $y\in \Omega _{b}$ and $n\geq 0$.\newline
(i) The series $\sum_{m=1}^{\infty }\rho _{\nu _{n},m}^{-2}u_{n,m,y}^{\delta
}$ converges uniformly on $\overline{\Omega _{b}}$ to a function $%
v_{n,y}^{\delta }$ which is harmonic on $A_{b}^{\prime }\times (\mathbb{R}%
\backslash \{y_{N}\})$ and continuously vanishes on $\partial \Omega _{b}$.%
\newline
(ii) $(-\Delta v_{n,y}^{\delta })(z)=h_{n,y^{\prime }}^{\delta }(z^{\prime
})dz^{\prime }d\eta _{y_{N}}$ on $\Omega _{b}$, in the sense of
distributions.
\end{lemma}

\begin{proof}
We know from Lemma \ref{hc} and the above remark that, inside $\Omega _{b}$,
the function $\rho _{\nu _{n},m}^{-2}u_{n,m,y}^{\delta }$ is the (Green)
potential of the measure 
\begin{equation*}
2a_{N}^{-1}\rho _{\nu _{n},m}^{-1}u_{n,m,y}^{\delta }(z^{\prime
},y_{N})dz^{\prime }d\eta _{y_{N}}.
\end{equation*}%
Since the potential 
\begin{equation*}
x\mapsto \int_{A_{b}^{\prime }}G_{\Omega _{b}}(x,(z^{\prime
},y_{N}))~dz^{\prime }\text{ \ \ \ }(x\in \Omega _{b})
\end{equation*}%
is bounded on $\Omega _{b}$, it only remains to note from Proposition \ref%
{Uexp} that the series 
\begin{equation*}
z^{\prime }\mapsto 2a_{N}^{-1}\sum_{m=1}^{\infty }\rho _{\nu
_{n},m}^{-1}u_{n,m,y}^{\delta }(z^{\prime },y_{N})
\end{equation*}%
converges uniformly on $A_{b}^{\prime }$ to $a_{N}^{-1}h_{n,y^{\prime
}}^{\delta }(z^{\prime })$.
\end{proof}

\begin{lemma}
\label{Lb}Let $y\in \Omega _{b}$ and $\delta \in (0,1)$. Then the series%
\begin{equation*}
x\mapsto \left\Vert y^{\prime }\right\Vert ^{\frac{3-N}{2}%
}\sum_{n=0}^{\infty }\nu _{n}v_{n,y}^{\delta }(x)\text{ \ \ \ }(N\geq 4),
\end{equation*}%
\begin{equation*}
x\mapsto \frac{1}{2}v_{0,y}^{\delta }(x)+\sum_{n=1}^{\infty }v_{n,y}^{\delta
}(x)\text{ \ \ \ }(N=3)
\end{equation*}%
converges uniformly on $\Omega _{b}$ to a function $g_{y}^{\delta }$ which
is the Green potential in $\Omega _{b}$ of the measure $G_{A_{b}^{\prime
}}\mu _{y^{\prime },\delta }(z^{\prime })dz^{\prime }d\eta _{y_{N}}$.
\end{lemma}

\begin{proof}
We know from Lemma \ref{La} that, inside $\Omega _{b}$, the function $%
v_{n,y}^{\delta }$ is the Green potential of the measure $%
a_{N}^{-1}h_{n,y^{\prime }}^{\delta }(z^{\prime })dz^{\prime }d\eta _{y_{N}}$%
. Further, by Proposition \ref{B} and Remark \ref{BR}, the series%
\begin{equation*}
z^{\prime }\mapsto \left\Vert y^{\prime }\right\Vert ^{\frac{3-N}{2}%
}\sum_{n=0}^{\infty }\nu _{n}h_{n,y}^{\delta }(z^{\prime })\text{ \ \ \ }%
(N\geq 4),
\end{equation*}%
\begin{equation*}
z^{\prime }\mapsto \frac{1}{2}h_{0,y}^{\delta }(z^{\prime
})+\sum_{n=1}^{\infty }h_{n,y}^{\delta }(z^{\prime })\text{ \ \ \ }(N=3)
\end{equation*}%
converge uniformly on $A_{b}^{\prime }$ to $G_{A_{b}^{\prime }}\mu
_{y^{\prime },\delta }(z^{\prime })$. This establishes the result.
\end{proof}

\bigskip

\begin{theorem}
\label{fin}If $y\in \Omega _{b}$ and $x\in A_{b}^{\prime }\times (\mathbb{R}%
\backslash \{y_{N}\})$, then 
\begin{equation}
G_{\Omega _{b}}(x,y)=\left\{ 
\begin{array}{cc}
\dfrac{a_{N}}{a_{N-1}}\left\Vert y^{\prime }\right\Vert ^{\frac{3-N}{2}%
}\dsum\limits_{n=0}^{\infty }\nu _{n}\dsum\limits_{m=1}^{\infty }u_{n,m,y}(x)
& (N\geq 4) \\ 
\text{ } & \text{ } \\ 
\dsum\limits_{m=1}^{\infty }u_{0,m,y}(x)+2\dsum\limits_{n=1}^{\infty
}\dsum\limits_{m=1}^{\infty }u_{n,m,y}(x) & (N=3)%
\end{array}%
\right. .  \label{GOb}
\end{equation}
\end{theorem}

\begin{proof}
Let $g_{y}^{\delta }$ be as in Lemma \ref{Lb}. By Lemma \ref{est}(i), Remark %
\ref{del} and Proposition \ref{Unu}(i) we can differentiate term-by-term to
see that 
\begin{equation}
\frac{\partial ^{2}g_{y}^{\delta }}{\partial x_{N}^{2}}=\left\{ 
\begin{array}{cc}
\left\Vert y^{\prime }\right\Vert ^{\frac{3-N}{2}}\sum\limits_{n=0}^{\infty
}\nu _{n}\sum\limits_{m=1}^{\infty }u_{n,m,y}^{\delta }(x) & (N\geq 4) \\ 
\text{ } & \text{ } \\ 
\frac{1}{2}\sum\limits_{m=1}^{\infty }u_{0,m,y}^{\delta
}(x)+\sum\limits_{n=1}^{\infty }\sum\limits_{m=1}^{\infty }u_{n,m,y}^{\delta
}(x) & (N=3)%
\end{array}%
\right.  \label{dg}
\end{equation}%
on $A_{b}^{\prime }\times (\mathbb{R}\backslash \{y_{N}\})$. The same
estimates show that the above function has limit $0$ on approach to $\infty $
within $\Omega _{b}$. By Lemma \ref{Lb}, we know that $\partial
^{2}g_{y}^{\delta }/\partial x_{N}^{2}$ is harmonic on $A_{b}^{\prime
}\times (\mathbb{R}\backslash \{y_{N}\})$ and vanishes on $\partial
A_{b}^{\prime }\times (\mathbb{R}\backslash \{y_{N}\})$. Further, 
\begin{equation*}
-\Delta \frac{\partial ^{2}g_{y}^{\delta }}{\partial x_{N}^{2}}=(-\Delta
G_{A_{b}^{\prime }}\mu _{y^{\prime },\delta })d\eta _{y_{N}}=a_{N-1}d\mu
_{y^{\prime },\delta }d\eta _{y_{N}}
\end{equation*}%
in the sense of distributions, so 
\begin{equation*}
\frac{a_{N}}{a_{N-1}}\frac{\partial ^{2}g_{y}^{\delta }}{\partial x_{N}^{2}}%
=G_{\Omega _{b}}(\mu _{y^{\prime },\delta }\eta _{y_{N}})\text{ \ \ on \ \ }%
\Omega _{b}.
\end{equation*}%
Finally, as $\delta \rightarrow 0$, we note that $G_{\Omega _{b}}(\mu
_{y^{\prime },\delta }\eta _{y_{N}})\rightarrow G_{\Omega _{b}}(\cdot ,y)$,
and from (\ref{dg}) we see that $(a_{N}/a_{N-1})\partial ^{2}g_{y}^{\delta
}/\partial x_{N}^{2}$ converges to the right hand side of (\ref{GOb})
locally uniformly on $A_{b}^{\prime }\times (\mathbb{R}\backslash \{y_{N}\})$%
, since $a_{3}/a_{2}=2$.
\end{proof}

\begin{corollary}
\label{C}Let $y\in \Omega _{b}$, $b^{\prime }>b$, $\varepsilon >0$ and $%
\delta \in (0,1)$. Then $G_{\Omega _{b}}(\cdot ,y)$ has a harmonic extension 
$\widetilde{G}_{\Omega _{b}}(\cdot ,y)$ to the set%
\begin{equation*}
L=\left\{ (x^{\prime },x_{N}):x_{N}\neq y_{N},\left\Vert x^{\prime
}\right\Vert >e^{-\left\vert x_{N}-y_{N}\right\vert /b^{\prime }}\right\}
\end{equation*}%
which satisfies%
\begin{equation}
\left\vert \widetilde{G}_{\Omega _{b}}(x,y)\right\vert \leq C(N,b,b^{\prime
},\delta ,\varepsilon )d(y^{\prime })\text{ \ \ \ }(x\in L,\left\vert
x_{N}-y_{N}\right\vert >\varepsilon ,\left\Vert x^{\prime }\right\Vert
>\delta ).  \label{bdg}
\end{equation}
\end{corollary}

\begin{proof}
Theorem \ref{fin}, and the estimates in Lemma \ref{est}(ii) and Proposition %
\ref{Unu}(i), together show that $G_{\Omega _{b}}(\cdot ,y)$ has a harmonic
extension to the set $\left( \mathbb{R}^{N-1}\backslash \overline{B}^{\prime
}\right) \times (\mathbb{R}\backslash \{y_{N}\})$ that satisfies $\left\vert 
\widetilde{G}_{\Omega _{b}}(x,y)\right\vert \leq C(N,b,\varepsilon
)d(y^{\prime })$\ when $\left\vert x_{N}-y_{N}\right\vert >\varepsilon $.
Further, by Lemma \ref{est}(iii), we see that $G_{\Omega _{b}}(\cdot ,y)$
has a harmonic extension to $L$ that satisfies (\ref{bdg}).
\end{proof}

\bigskip

Theorem \ref{main} is an immediate consequence of the following result,
subject to the verification below of Theorem \ref{app}.

\begin{theorem}
Let $c>0$ and $h$ be a harmonic function on $A_{b}^{\prime }\times (-c,c)$
which continuously vanishes on $\partial A_{b}^{\prime }\times (-c,c)$. Then 
$h$ has a harmonic extension to the set 
\begin{equation*}
\left\{ (x^{\prime },x_{N}):\left\vert x_{N}\right\vert <c,\left\Vert
x^{\prime }\right\Vert >e^{(\left\vert x_{N}\right\vert -c)/b}\right\} .
\end{equation*}
\end{theorem}

\begin{proof}
Let $0<c^{\prime \prime }<c^{\prime }<c$ and $b^{\prime }>b$. On $%
A_{b}^{\prime }\times (-c^{\prime },c^{\prime })$ we can write $h$ as the
difference $h_{1}-h_{2}$ of two positive harmonic functions that vanish on $%
\partial A_{b}^{\prime }\times (-c^{\prime },c^{\prime })$. (We can write $h$
as the difference of two Dirichlet solutions there with non-negative
boundary data.) Next, let $h_{i}^{\ast }$ $(i=1,2)$ be defined as $h_{i}$ on 
$A_{b}^{\prime }\times \lbrack -c^{\prime \prime },c^{\prime \prime }]$, as $%
0$ on $\left( A_{b}^{\prime }\times (-\infty ,-c^{\prime }]\right) \cup
\left( A_{b}^{\prime }\times \lbrack c^{\prime },\infty )\right) \cup
\partial \Omega _{b}$, and extended to $\Omega _{b}$ by solving the
Dirichlet problem in $A_{b}^{\prime }\times (-c^{\prime },-c^{\prime \prime
})$ and in $A_{b}^{\prime }\times (c^{\prime \prime },c^{\prime })$. Then $%
h_{i}^{\ast }$ is subharmonic on $A_{b}^{\prime }\times \left( (-\infty
,-c^{\prime \prime })\cup (c^{\prime \prime },\infty )\right) $ and
superharmonic on $A_{b}^{\prime }\times (-c^{\prime },c^{\prime })$, and
continuously vanishes on $\partial \Omega _{b}$. We\ can write $h_{i}^{\ast
} $ as $G_{\Omega _{b}}\mu _{i}$, where $\mu _{i}$ is a signed measure on $%
\overline{A_{b}^{\prime }}\times \{\pm c^{\prime },\pm c^{\prime \prime }\}$
satisfying $\int d(y^{\prime })d\left\vert \mu _{i}\right\vert (y)<\infty $.

It follows from Corollary \ref{C} that the formula 
\begin{equation*}
\widetilde{h}(x)=\int_{\overline{A_{b}^{\prime }}\times \{\pm c^{\prime
},\pm c^{\prime \prime }\}}\widetilde{G}_{\Omega _{b}}(x,y)d(\mu _{1}-\mu
_{2})(y)
\end{equation*}%
defines a harmonic extension of $h$ from $A_{b}^{\prime }\times (-c^{\prime
\prime },c^{\prime \prime })$\ to the set 
\begin{equation*}
\left\{ x_{N}<c^{\prime \prime },\left\Vert x^{\prime }\right\Vert
>e^{(x_{N}-c^{\prime \prime })/b^{\prime }}\right\} \cap \left\{
x_{N}>-c^{\prime \prime },\left\Vert x^{\prime }\right\Vert
>e^{-(x_{N}+c^{\prime \prime })/b^{\prime }}\right\} .
\end{equation*}%
Since $c^{\prime \prime }$ can be arbitrarily close to $c$, and $b^{\prime }$
can be arbitrarily close to $b$, the result follows.
\end{proof}

\section{Proof of Theorem \protect\ref{app}\label{appbegin}}

Let 
\begin{equation*}
u_{\nu }\left( x,y\right) =x^{1/2}J_{\nu }\left( y\right) Y_{\nu }\left(
xy\right) -x^{1/2}J_{\nu }\left( xy\right) Y_{\nu }\left( y\right) \text{ \
\ \ }(x>0,y>0).
\end{equation*}%
We know that the cylinder function $x\longmapsto u_{\nu }\left( x,y\right) $
has infinitely many positive zeros which are all simple (see Sections 15.21,
15.24 of \cite{Wat}). Let $x_{\nu ,k}\left( y\right) $ denote the $k$th zero
of this function in $(1,\infty )$. By Lemma \ref{JV}(vi) and Sturm's
comparison theorem \cite{Dur},%
\begin{equation}
x_{\nu ,k+1}\left( y\right) -x_{\nu ,k}\left( y\right) \geq \frac{\pi }{y}%
\text{ \ \ \ }\left( \nu \geq \frac{1}{2}\right) .  \label{eqZeroDiff}
\end{equation}%
Further, by Sturm's convexity theorem \cite{Dur},%
\begin{equation}
x_{\nu ,k+1}\left( y\right) -x_{\nu ,k}\left( y\right) <x_{\nu ,k}\left(
y\right) -x_{\nu ,k-1}\left( y\right) \text{ \ \ \ }\left( k\geq 2,\nu \geq 
\frac{1}{2}\right) .  \label{SCT}
\end{equation}

We collect together some useful facts about $u_{\nu }\left( x,y\right) $
below.

\begin{lemma}
\label{uxy}If $u_{\nu }\left( x_{0},y_{0}\right) =0$, where $x_{0}>1$, then 
\begin{equation}
\frac{\partial u_{\nu }}{\partial x}\left( x_{0},y_{0}\right) \frac{\partial
u_{\nu }}{\partial y}\left( x_{0},y_{0}\right)
=2y_{0}\int_{1}^{x_{0}}\left\{ u\left( x,y_{0}\right) \right\} ^{2}dx>0,
\label{equxuy}
\end{equation}%
\begin{equation}
\frac{\partial u_{\nu }/\partial x}{\partial u_{\nu }/\partial y}\left(
x_{0},y_{0}\right) >\frac{y_{0}}{x_{0}},  \label{eqbasic1}
\end{equation}%
\begin{equation}
\left( \dfrac{\partial u_{\nu }}{\partial x}\left( x\dfrac{\partial u_{\nu }%
}{\partial x}-y\dfrac{\partial u_{\nu }}{\partial y}\right) \right) \left(
x_{0},y_{0}\right) =\frac{4}{\pi ^{2}},  \label{wronk}
\end{equation}%
\begin{equation}
2x_{0}\frac{\partial ^{2}u_{\nu }}{\partial y\partial x}\left(
x_{0},y_{0}\right) =2\dfrac{\partial u_{\nu }}{\partial y}\left(
x_{0},y_{0}\right) +y_{0}\frac{\partial ^{2}u_{\nu }}{\partial y^{2}}\left(
x_{0},y_{0}\right) .  \label{P13}
\end{equation}
\end{lemma}

\begin{proof}
Inequality (\ref{equxuy}) is known \cite{Coc}, but we will give a short
alternative proof. We abbreviate $u_{\nu }$ to $u$, define $q(x,y)=\left(
\nu ^{2}-\frac{1}{4}\right) x^{-2}-y^{2}$, and note from Lemma \ref{JV}(vi)
that $u_{xx}=qu$. Hence $u_{xxy}=qu_{y}+q_{y}u$, and so 
\begin{equation*}
\frac{\partial }{\partial x}\left( u_{xy}u-u_{x}u_{y}\right)
=u_{xxy}u-u_{xx}u_{y}=(q_{y}u+qu_{y})u-quu_{y}=q_{y}u^{2}=-2yu^{2}.
\end{equation*}%
Since $u(1,\cdot )\equiv 0$, we can set $y=y_{0}$ and integrate the above
equation with respect to $x$ over $[1,x_{0}]$ to obtain (\ref{equxuy}).

We fix $y$ and define 
\begin{equation*}
f\left( x\right) =2y^{2}\int_{1}^{x}\left\{ u\left( t,y\right) \right\}
^{2}dt+xqu^{2}-x\left( u_{x}\right) ^{2}+u_{x}u.
\end{equation*}%
Using the fact that $u_{xx}=qu$, we obtain%
\begin{equation*}
f^{\prime }\left( x\right) =u^{2}\left( 2y^{2}+2q+xq_{x}\right)
=u^{2}\left\{ 2y^{2}+2\left( \frac{\nu ^{2}-\frac{1}{4}}{x^{2}}-y^{2}\right)
-2\frac{\nu ^{2}-\frac{1}{4}}{x^{2}}\right\} =0.
\end{equation*}%
Since $u(1,\cdot )\equiv 0$, we see that $f\equiv f(1)=-\left(
u_{x}(1,y)\right) ^{2}$. Further, 
\begin{equation}
u_{x}(1,y)=y\left( J_{\nu }\left( y\right) Y_{\nu }^{\prime }\left( y\right)
-J_{\nu }^{\prime }\left( y\right) Y_{\nu }\left( y\right) \right) =\frac{2}{%
\pi },  \label{ux}
\end{equation}%
by Lemma \ref{JV}(iii). Since $u\left( x_{0},y_{0}\right) =0$, we conclude
that 
\begin{equation*}
2y_{0}^{2}\int_{1}^{x_{0}}\left\{ u\left( t,y_{0}\right) \right\}
^{2}dt-x_{0}\left( \frac{\partial u}{\partial x}\left( x_{0},y_{0}\right)
\right) ^{2}=f(x_{0})=f(1)=-\frac{4}{\pi ^{2}}<0.
\end{equation*}%
Thus, by (\ref{equxuy}), 
\begin{equation*}
y_{0}\frac{\partial u}{\partial x}\left( x_{0},y_{0}\right) \frac{\partial u%
}{\partial y}\left( x_{0},y_{0}\right) =2y_{0}^{2}\int_{1}^{x_{0}}\left\{
u\left( t,y_{0}\right) \right\} ^{2}dt<x_{0}\left( \frac{\partial u}{%
\partial x}\left( x_{0},y_{0}\right) \right) ^{2},
\end{equation*}%
and (\ref{eqbasic1}) follows.

Let $w=xu_{x}-yu_{y}$. Direct computation shows that $w_{xx}=qw$, since $%
2q+xq_{x}-yq_{y}=0$. For any fixed value of $y$ the expression $%
uw_{x}-u_{x}w $ thus has a constant value. Since $u(1,\cdot )\equiv 0$, and $%
w\left( 1,y\right) =u_{x}\left( 1,y\right) =2/\pi $ by (\ref{ux}), we
conclude that%
\begin{equation}
uw_{x}-u_{x}w=-\frac{4}{\pi ^{2}},  \label{ai}
\end{equation}%
which yields (\ref{wronk}) because $u(x_{0},y_{0})=0$.

Differentiation of (\ref{ai}) with respect to $y$ yields 
\begin{eqnarray*}
0 &=&u_{y}w_{x}+uw_{xy}-u_{xy}w-u_{x}w_{y} \\
&=&u_{y}(u_{x}+xu_{xx}-yu_{xy})+uw_{xy}-u_{xy}\left( xu_{x}-yu_{y}\right)
-u_{x}\left( xu_{xy}-u_{y}-yu_{yy}\right) \\
&=&2u_{y}u_{x}+xu_{xx}u_{y}+uw_{xy}-2u_{xy}xu_{x}+u_{x}yu_{yy} \\
&=&u\left( xqu_{y}+w_{xy}\right) +u_{x}\left( 2u_{y}-2xu_{xy}+yu_{yy}\right)
,
\end{eqnarray*}%
since $u_{xx}=qu$. This simplifies to (\ref{P13}) because $u\left(
x_{0},y_{0}\right) =0$ and $u_{x}(x_{0},y_{0})\neq 0$.
\end{proof}

\bigskip

If $x>1$, then we see from (\ref{equxuy}) that the function $y\longmapsto
u_{\nu }\left( x,y\right) $ has only simple zeros on $(0,\infty )$. We
define $y_{\nu ,k}\left( x\right) >0$ to be the $k$th positive zero. (When $%
x=b$ these correspond to the zeros $\rho _{\nu ,k}$ defined in Section 2.)
Further, in view of (\ref{equxuy}), the implicit function theorem can be
applied to the function $u_{\nu }:\left( 1,\infty \right) \times \left(
0,\infty \right) \rightarrow \mathbb{R}$ to see that $y_{\nu ,k}$ is
differentiable on $\left( 1,\infty \right) $, and so we can differentiate
the equation $u_{\nu }\left( x,y_{\nu ,k}\left( x\right) \right) =0$ to
obtain 
\begin{equation}
y_{\nu ,k}^{\prime }\left( x\right) =-\dfrac{\partial u_{\nu }}{\partial x}%
\left( x,y_{\nu ,k}\left( x\right) \right) /\dfrac{\partial u_{\nu }}{%
\partial y}\left( x,y_{\nu ,k}\left( x\right) \right) <0,  \label{eqykprime}
\end{equation}%
by (\ref{equxuy}) again, whence $y_{\nu ,k}$ is strictly decreasing on $%
\left( 1,\infty \right) $. The following simple observation will help us to
show that $y_{\nu ,k}$ is also convex.

\begin{lemma}
\label{simple}Suppose that $u\left( x,y\right) $ is a function such that $%
u_{xx}=qu$, and $y_{k}$ is a differentiable function such that $u\left(
x,y_{k}\left( x\right) \right) =0$. If $%
u_{x}u_{y}(2u_{xy}u_{y}-u_{yy}u_{x})>0$ on the zero set of $u$, then $y_{k}$
is convex.
\end{lemma}

\begin{proof}
We know that $y_{k}^{\prime }u_{y}\left( x,y_{k}\left( x\right) \right)
=-u_{x}\left( x,y_{k}\left( x\right) \right) $, and $u_{xx}=qu=0$ on the
zero set of $u$, so%
\begin{eqnarray*}
y_{k}^{\prime \prime }\left( x\right) &=&-\frac{d}{dx}\left( \frac{u_{x}}{%
u_{y}}\left( x,y_{k}\left( x\right) \right) \right) =-\frac{%
u_{xy}y_{k}^{\prime }u_{y}-u_{x}\left( u_{xy}+u_{yy}y_{k}^{\prime }\right) }{%
u_{y}^{2}}\left( x,y_{k}\left( x\right) \right) \\
&=&\left( \frac{u_{x}}{u_{y}}\frac{2u_{xy}u_{y}-u_{yy}u_{x}}{u_{y}^{2}}%
\right) \left( x,y_{k}\left( x\right) \right) >0.
\end{eqnarray*}
\end{proof}

\bigskip

The following result will be proved in Section \ref{append}.

\begin{proposition}
\label{ThmDE} Let $\nu \geq 
{\frac12}%
$. Then, for each $x>1$ the cross product $y\longmapsto u_{\nu }\left(
x,y\right) $ satisfies a second order differential equation, $\widetilde{P}%
F^{\prime \prime }-\widetilde{P}^{\prime }F^{\prime }+\widetilde{Q}F=0$,
where $\widetilde{P}\left( x,y\right) >0$ and $\widetilde{P}^{\prime }\left(
x,y\right) <0$.
\end{proposition}

We now prove a result that contains Theorem \ref{app}.

\begin{theorem}
\label{ThmMain2}If $\nu \geq 
{\frac12}%
$,\ then the zero curves $y_{\nu ,k}:\left( 1,\infty \right) \rightarrow
\left( 0,\infty \right) $ are convex, and 
\begin{equation}
y_{\nu ,k+1}\left( x\right) -y_{\nu ,k}\left( x\right) >\frac{\pi }{2x-1}%
\text{ \ \ \ }\left( k\geq 2\right) .  \label{spac}
\end{equation}
\end{theorem}

\begin{proof}
On the zero set of $u_{\nu }$ we have, by (\ref{P13}) and then (\ref{wronk}%
), 
\begin{eqnarray*}
x\left( 2\frac{\partial ^{2}u_{\nu }}{\partial x\partial y}\frac{\partial
u_{\nu }}{\partial y}-\frac{\partial ^{2}u_{\nu }}{\partial y^{2}}\frac{%
\partial u_{\nu }}{\partial x}\right) &=&2\left( \frac{\partial u_{\nu }}{%
\partial y}\right) ^{2}+\frac{\partial ^{2}u_{\nu }}{\partial y^{2}}\left( y%
\frac{\partial u_{\nu }}{\partial y}-x\frac{\partial u_{\nu }}{\partial x}%
\right) \\
&=&2\left( \frac{\partial u_{\nu }}{\partial y}\right) ^{2}-\frac{\partial
^{2}u_{\nu }}{\partial y^{2}}\frac{4}{\pi ^{2}}\left\{ \frac{\partial u_{\nu
}}{\partial x}\right\} ^{-1},
\end{eqnarray*}%
whence%
\begin{equation*}
\frac{\partial u_{\nu }}{\partial x}\frac{\partial u_{\nu }}{\partial y}%
\left( 2\frac{\partial ^{2}u_{\nu }}{\partial x\partial y}\frac{\partial
u_{\nu }}{\partial y}-\frac{\partial ^{2}u_{\nu }}{\partial y^{2}}\frac{%
\partial u_{\nu }}{\partial x}\right) =\frac{2}{x}\frac{\partial u_{\nu }}{%
\partial x}\left( \frac{\partial u_{\nu }}{\partial y}\right) ^{3}-\frac{%
\partial ^{2}u_{\nu }}{\partial y^{2}}\frac{\partial u_{\nu }}{\partial y}%
\frac{4}{\pi ^{2}x}.
\end{equation*}%
The first term on the right hand side is positive, by (\ref{equxuy}), and
the second is negative, by Proposition \ref{ThmDE}. Hence $y_{k}$ is convex,
by Lemma \ref{simple}.

Let $y_{0}=y_{\nu ,k}\left( x_{0}\right) $ be given, where $x_{0}>1$. Then $%
x_{0}$ is the $k$th zero of $x\longmapsto u_{\nu }\left( x,y_{0}\right) $ in 
$(1,\infty )$, so $x_{0}=x_{\nu ,k}\left( y_{0}\right) .$ We now consider
the next zero, $x_{\nu ,k+1}\left( y_{0}\right) $. By the convexity of $%
y_{\nu ,k+1}$,%
\begin{equation*}
y_{\nu ,k}\left( x_{0}\right) +\left\{ x_{\nu ,k}(y_{0})-x_{\nu
,k+1}(y_{0})\right\} y_{\nu ,k+1}^{\prime }\left( x_{\nu ,k+1}(y_{0})\right)
\leq y_{\nu ,k+1}\left( x_{0}\right) .
\end{equation*}%
We use (\ref{eqykprime}), (\ref{eqbasic1}) and (\ref{eqZeroDiff}) to deduce
that 
\begin{eqnarray*}
y_{\nu ,k+1}\left( x_{0}\right) -y_{\nu ,k}\left( x_{0}\right) &\geq &\left(
x_{\nu ,k+1}\left( y_{0}\right) -x_{\nu ,k}\left( y_{0}\right) \right) \frac{%
\dfrac{\partial u_{\nu }}{\partial x}\left( x_{\nu ,k+1}\left( y_{0}\right)
,y_{0}\right) }{\dfrac{\partial u_{\nu }}{\partial y}\left( x_{\nu
,k+1}\left( y_{0}\right) ,y_{0}\right) } \\
&\geq &\left( x_{\nu ,k+1}\left( y_{0}\right) -x_{\nu ,k}\left( y_{0}\right)
\right) \frac{y_{0}}{x_{\nu ,k+1}(y_{0})} \\
&\geq &\frac{\pi }{x_{\nu ,k+1}(y_{0})}.
\end{eqnarray*}%
Finally, by (\ref{SCT}), 
\begin{equation*}
x_{\nu ,k+1}\left( y_{0}\right) <2x_{\nu ,k}\left( y_{0}\right) -x_{\nu
,k-1}\left( y_{0}\right) <2x_{\nu ,k}\left( y_{0}\right) -1=2x_{0}-1\text{\
\ \ }(k\geq 2),
\end{equation*}%
so we arrive at (\ref{spac}).
\end{proof}

\section{Proof of Proposition \protect\ref{ThmDE}\label{append}}

Let 
\begin{equation}
F\left( y\right) =a\left( y\right) f\left( y\right) +b\left( y\right)
g\left( y\right) ,  \label{eqFF}
\end{equation}%
where 
\begin{equation}
f\left( y\right) =Y_{\nu }\left( xy\right) ,\text{ \ }g\left( y\right)
=-J_{\nu }\left( xy\right) ,\text{ \ }a\left( y\right) =J_{\nu }\left(
y\right) ,\text{ \ }b\left( y\right) =Y_{\nu }\left( y\right) ,  \label{Bes}
\end{equation}%
and $x>1$ is fixed. We will show that functions of the form (\ref{eqFF})
satisfy a certain second order differential equation, and that when (\ref%
{Bes}) holds the signs of the coefficients in this equation are as described
in Proposition \ref{ThmDE}.

Let $f,g,a,b$ be differentiable functions defined on $\left( c,d\right) $,
and let%
\begin{equation*}
W=fg^{\prime }-f^{\prime }g,\text{ \ \ }N=f^{2}+g^{2},\text{ \ \ }%
w=ab^{\prime }-a^{\prime }b\text{, \ \ }n=a^{2}+b^{2}\text{.}
\end{equation*}

\begin{lemma}
\label{PropABCD}If $F=af+bg$, then%
\begin{equation*}
F^{\prime }=Af+Bg\text{ \ and \ }F^{\prime \prime }=Cf+Dg,
\end{equation*}%
where 
\begin{eqnarray*}
A &=&a^{\prime }+\frac{aN^{\prime }}{2N}+\frac{bW}{N}\text{ \ and }%
B=b^{\prime }+\frac{bN^{\prime }}{2N}-\frac{aW}{N}, \\
C &=&A^{\prime }+\frac{AN^{\prime }}{2N}+\frac{BW}{N}\text{ \ and \ }%
D=B^{\prime }+\frac{BN^{\prime }}{2N}-\frac{AW}{N}.
\end{eqnarray*}
\end{lemma}

\begin{proof}
Since 
\begin{equation*}
2Nf^{\prime }-N^{\prime }f=2g^{2}f^{\prime }-2g^{\prime }gf=-2gW\text{ \ and
\ }2Ng^{\prime }-N^{\prime }g=2f^{2}g^{\prime }-2f^{\prime }fg=2fW,
\end{equation*}%
we see that%
\begin{equation*}
f^{\prime }=\frac{N^{\prime }}{2N}f-\frac{W}{N}g\text{ \ \ and \ \ }%
g^{\prime }=\frac{N^{\prime }}{2N}g+\frac{W}{N}f,
\end{equation*}%
and so 
\begin{equation*}
F^{\prime }=a^{\prime }f+af^{\prime }+b^{\prime }g+bg^{\prime }=f\left(
a^{\prime }+\frac{aN^{\prime }}{2N}+\frac{bW}{N}\right) +g\left( b^{\prime }+%
\frac{bN^{\prime }}{2N}-\frac{aW}{N}\right) .
\end{equation*}%
Thus $F^{\prime }=Af+Bg$. The same reasoning, applied to $F^{\prime }$,
shows that$\ F^{\prime \prime }=Cf+Dg$.
\end{proof}

\begin{proposition}
Let $f,g,a,b$ be smooth functions. Then the function $F=af+bg$ satisfies the
differential equation%
\begin{equation}
\widetilde{P}F^{\prime \prime }-\widetilde{P}^{\prime }F^{\prime }+%
\widetilde{Q}F=0,  \label{eqDEQ}
\end{equation}%
where $\widetilde{P}=Wn-Nw$ and $\widetilde{Q}=N\left( CB-DA\right) $.
\end{proposition}

\begin{proof}
We know from Lemma \ref{PropABCD} that $F^{\prime }=Af+Bg$ and $F^{\prime
\prime }=Cf+Dg.$ Hence 
\begin{equation}
\left( Ab-Ba\right) F^{\prime \prime }-\left( Cb-Da\right) F^{\prime
}+\left( CB-DA\right) F=0,  \label{eqDEG2}
\end{equation}%
because (trivially) 
\begin{eqnarray*}
\left( Ab-Ba\right) C-\left( Cb-Da\right) A+\left( CB-DA\right) a &=&0, \\
\left( Ab-Ba\right) D-\left( Cb-Da\right) B+\left( CB-DA\right) b &=&0.
\end{eqnarray*}%
Since%
\begin{equation}
Ab-Ba=a^{\prime }b-b^{\prime }a+\frac{W}{N}\left( b^{2}+a^{2}\right) ,
\label{eqAbBa}
\end{equation}%
we see that 
\begin{equation*}
N\left( Ab-Ba\right) =W\left( a^{2}+b^{2}\right) -N\left( ab^{\prime
}-a^{\prime }b\right) =\widetilde{P},
\end{equation*}%
and we can multiply across (\ref{eqDEG2}) by $N$ to get 
\begin{equation*}
\widetilde{P}F^{\prime \prime }-N\left( Cb-Da\right) F^{\prime }+N\left(
CB-DA\right) F=0.
\end{equation*}

It remains to check that $N\left( Cb-Da\right) =\widetilde{P}^{\prime }$.
Since 
\begin{equation*}
N\left( Cb-Da\right) =N\left( A^{\prime }b-B^{\prime }a\right) +\frac{%
N^{\prime }}{2}\left( Ab-Ba\right) +W\left( Bb+Aa\right)
\end{equation*}%
and 
\begin{equation*}
\widetilde{P}^{\prime }=N^{\prime }\left( Ab-Ba\right) +N\left( A^{\prime
}b-B^{\prime }a\right) +N\left( Ab^{\prime }-Ba^{\prime }\right) ,
\end{equation*}%
we see that $N\left( Cb-Da\right) -\widetilde{P}^{\prime }=-\Delta _{1}$,
where 
\begin{equation*}
\Delta _{1}=\frac{N^{\prime }}{2}\left( Ab-Ba\right) +N\left( Ab^{\prime
}-Ba^{\prime }\right) -W\left( Bb+Aa\right) ,
\end{equation*}%
and it suffices to show that $\Delta _{1}=0$. We compute 
\begin{eqnarray*}
Ab^{\prime }-Ba^{\prime } &=&\frac{N^{\prime }}{2N}\left( ab^{\prime
}-a^{\prime }b\right) +\frac{W}{N}\left( a^{\prime }a+b^{\prime }b\right) ,
\\
Bb+Aa &=&b^{\prime }b+a^{\prime }a+\frac{N^{\prime }}{2N}\left(
b^{2}+a^{2}\right) ,
\end{eqnarray*}%
and use these identities along with (\ref{eqAbBa})\ to obtain 
\begin{eqnarray*}
\Delta _{1} &=&\frac{N^{\prime }}{2}\left( a^{\prime }b-b^{\prime }a\right) +%
\frac{N^{\prime }}{2}\frac{W}{N}\left( a^{2}+b^{2}\right) +\frac{N^{\prime }%
}{2}\left( ab^{\prime }-a^{\prime }b\right) +W\left( a^{\prime }a+b^{\prime
}b\right) \\
&&\text{ }-W\left( b^{\prime }b+a^{\prime }a+\frac{N^{\prime }}{2N}\left(
b^{2}+a^{2}\right) \right) =0.
\end{eqnarray*}
\end{proof}

\bigskip

\begin{proof}[Proof of Proposition \protect\ref{ThmDE}]
We apply the preceding proposition to the case where (\ref{Bes}) holds. Then 
$N(y)=N_{\nu }\left( xy\right) $, $n\left( y\right) =N_{\nu }\left( y\right) 
$, 
\begin{equation*}
W\left( y\right) =xJ_{\nu }\left( xy\right) Y_{\nu }^{\prime }\left(
xy\right) -xJ_{\nu }^{\prime }\left( xy\right) Y_{\nu }\left( xy\right) =x%
\frac{2}{\pi xy}=\frac{2}{\pi y}
\end{equation*}%
by Lemma \ref{JV}(iii), and similarly $w\left( y\right) =2/(\pi y)$.\
Further, 
\begin{equation*}
\widetilde{P}\left( x,y\right) =Wn-wN=\frac{2}{\pi y}\left( N_{\nu }\left(
y\right) -N_{\nu }\left( xy\right) \right) >0
\end{equation*}%
by Lemma \ref{JV}(v), and 
\begin{equation*}
\widetilde{P}^{\prime }\left( x,y\right) =-\frac{\widetilde{P}\left(
x,y\right) }{y}+\frac{2}{\pi y^{2}}\left( yN_{\nu }^{\prime }\left( y\right)
-yxN_{\nu }^{\prime }\left( xy\right) \right) <0,
\end{equation*}%
since $y\longmapsto yN_{\nu }^{\prime }(y)$ is increasing. (It is clear from
p.446 of \cite{Wat} that $(d/dt)(tN_{\nu }(t))$ is increasing when $\nu \geq 
\frac{1}{2}$, and we also know that $N_{\nu }$ is decreasing.) Proposition %
\ref{ThmDE} is now established, because $u_{\nu }(x,y)=\sqrt{x}F(y)$.
\end{proof}

\bigskip

\noindent \textit{Stephen J. Gardiner}

\noindent School of Mathematics and Statistics,

University College Dublin,

Belfield, Dublin 4, Ireland.

email: stephen.gardiner@ucd.ie

\bigskip

\noindent \textit{Hermann Render}

\noindent School of Mathematics and Statistics,

University College Dublin,

Belfield, Dublin 4, Ireland.

email: hermann.render@ucd.ie

\end{document}